\documentclass{amsart}
\usepackage{amsmath,amssymb,amsthm,booktabs,array,longtable}
\usepackage{amssymb}
\usepackage{verbatim}
\usepackage{array}
\usepackage{latexsym}
\usepackage{enumerate}
\usepackage{amsmath}
\usepackage{amsfonts}
\usepackage{amsthm}
\usepackage{color}
\usepackage[english]{babel}
\usepackage{comment}

\newtheorem{theorem}{Theorem}[section]
\newtheorem{proposition}[theorem]{Proposition}
\newtheorem{lemma}[theorem]{Lemma}

\newtheorem{result}[theorem]{Result}

{\theoremstyle{definition}
\newtheorem{remark}[theorem]{Remark}

}

\def\cH{\mathcal H}

\def\cM{\mathcal M}

\def\cQ{\mathcal Q}

\def\cX{\mathcal X}
\def\cY{\mathcal Y}

\def\K{\mathbb{K}}

\def\ord{\mbox{\rm ord}}
\def\deg{\mbox{\rm deg}}
\def\det{\mbox{\rm det}}

\def\min{{\rm min}}

\def\gg{\mathfrak{g}}

\newcommand{\Spi}[1]{S_{P}^{(#1)}}
\newcommand{\ab}{(\mu+1)}

\newcommand{\PSL}{\mbox{\rm PSL}}
\newcommand{\PGL}{\mbox{\rm PGL}}

\newcommand{\PGU}{\mbox{\rm PGU}}

\newcommand{\aut}{\mbox{\rm Aut}}
\newcommand{\gal}{\mbox{\rm Gal}}

\newcommand{\ga}{\alpha}
\newcommand{\gb}{\beta}

\newcommand{\gl}{\lambda}

\newcommand{\gt}{\tau}

\newcommand{\ha}{{\textstyle\frac{1}{2}}}

\title{Large automorphism groups compared to the $p$-rank of algebraic curves in characteristic $p$ \\\bigskip
}
\thanks{The
research  was supported by
 the Italian National
Group for Algebraic and Geometric Structures and their Applications (GNSAGA
- INdAM)}

\author{Massimo Giulietti}
\address{Massimo Giulietti\\              Dipartimento di Matematica e Informatica - Universit\`a degli Studi di Perugia\\ Via Vanvitelli, 1 - 06123 Perugia (Italy)\\
}
               \email{massimo.giulietti@unipg.it}

\author{G\'abor Korchm\'aros}
\address{G\'abor Korchm\'aros\\
 Dipartimento di Scienze di Base e Applicazioni - Universit\`a della Basilicata\\ Contrada Macchia
Romana - 85100 Potenza (Italy)\\
               }
              \email{gabor.korchmaros@unibas.it}
\author{Marco Timpanella}
\address{Marco Timpanella\\              Dipartimento  di Matematica e Informatica- Universit\`a degli Studi di Perugia\\ Via Vanvitelli, 1 - 06123 Perugia (Italy)\\
}
               \email{marco.timpanella@unipg.it}

\date{}

\begin{document}

\begin{abstract}
Let $\cX$ be a (projective, geometrically irreducible, non-singular) algebraic curve of genus $\ge 2$ and positive $p$-rank $\gamma(\cX)$, defined over an algebraically closed field $\mathbb{K}$ of positive characteristic $p>0$. Contrary to what occurs for the genus, no function $h(\gamma)$ exists such that $|\aut(\cX)|\le h(\gamma)$ whenever $\gamma=\gamma(\cX)$. Thus, to have a bound on $|\aut(\cX)|$ only depending on $\gamma(\cX)$, some restrictions on $\cX$ and $\aut(\cX)$ are needed. In this context, the following theorem is proven.  Let $\Gamma$ be a subgroup of $\aut(\cX)$.
Assume that there is a point $P\in \cX$ such that
 the quotient curve $\cX/S_P$ is rational,
where $S_P$ denotes the Sylow p-subgroup of the stabilizer $\Gamma_P$ of $P$ in $\Gamma$.
Then
the following $p$-rank analog of the Riemann-Hurwitz bound
\begin{equation*}
|\Gamma|\le 24 \left(\frac{p}{p-1}\right)^4 \gamma(\cX)^4
\end{equation*}
holds, unless a subgroup of index $\le 2$ of $\Gamma$ fixes $P$. This bound is sharp up to a constant depending only on $p$.
\end{abstract}
\maketitle
\vspace{0.5cm}\noindent {\em Keywords}:
algebraic curve, function field, positive characteristic, automorphism group, $p$-rank.
\vspace{0.2cm}\noindent

\vspace{0.5cm}\noindent {\em Subject classifications}:
\vspace{0.2cm}\noindent  14H37, 14H05.

\section{Introduction}
The classical theory of birational invariants of algebraic curves was developed over time, originating from Riemann's work on complex curves (Riemann surfaces) in the mid-19th century. The most important birational invariants of a (projective, geometrically irreducible, non-singular, algebraic) curve $\cX$ defined over a field $\mathbb{K}$ include its genus $\gg(\cX)$ and the order of its automorphism group $\aut(\cX)$ fixing $\mathbb{K}$ elementwise. The earliest major achievement was the Riemann-Hurwitz formula yielding the upper bound  $84(\gg(\cX)-1)$ on the order of the automorphism group of any complex curve other than rational and elliptic curves. Ever since, comparison of numerical birational invariants of curves has attracted remarkable interest, and the extensive literature also contains contributions that may be viewed as refinements of the Riemann-Hurwitz bound for automorphism groups with special structures, such as cyclic groups and groups whose order is a prime-power; see \cite{BCG,HRA,ZR} and the references therein.

The classical results on birational invariants hold true for curves defined over any field of zero characteristic, whereas the picture changes radically whenever the characteristic is assumed to be positive. A first sign is that the classical Riemann-Hurwitz bound fails. Although the finiteness of the automorphism group of any curve other than the rational and elliptic remains valid, nevertheless the Riemann-Hurwitz type upper bound on $|\aut(\cX)|$ increases to $16\gg(\cX)^4$, with the unique exception of the Hermitian curve of genus $\ha q(q-1)$ whose automorphism group is isomorphic to $\PGU(3,q)$ and hence it has order $(q^3+1)q^3(q^2-1)$ where $q$ is any power of the characteristic; see \cite{stichtenoth1973I,stichtenoth1973II}.

As a matter of fact, in positive characteristic, automorphism groups of curves present several different features in the non-tame case which occurs when the curve $\cX$ contains a point $P$ such that the stabilizer  $\Gamma_P$ of $P$ in $\Gamma=\aut(\cX)$ has a non-trivial $p$-subgroup.  A possible cause and explanation of this phenomenon may be the structural difference in $\Gamma_P$, since $\Gamma_P$ is cyclic in zero characteristic while it is the semidirect product of $S_P$ with a cyclic group where $S_P$ is the unique (normal) Sylow $p$-subgroup of $\Gamma_P$ whose structure varies depending on the local properties of the curve. This is apparent in ramification theory and its applications.

In characteristic $p>0$, a third relevant numerical birational invariant is the $p$-rank (equivalently the Hasse-Witt invariant) $\gamma(\cX)$ of a curve $\cX$ where $0\le \gamma(\cX)\le \gg(\cX)$.
From previous work it emerged that curves with large automorphism groups have zero $p$-rank, and this raised the problem of determining a function $h(\gg)$ defined for $\gg\ge 2$ such that if $|G|>h(\gg)$ for a subgroup $G\le \aut(\cX)$ of a curve $\cX$ of genus $\gg=\gg(\cX)$ then $\gamma(\cX)=0$. Such a function exists, as Henn's classification \cite{henn1978} shows that there are only four curves with $|\aut(\cX)|\ge 8\gg(\cX)^3$ and each of them has zero $p$-rank. It seems plausible that the amplitude of $h(\gg)$ may be much smaller, namely $h(\gg)\approx c\gg^2$. So far, this has been proven in odd characteristic: for solvable groups $G$ with $c=126$, and for non-solvable groups $G$ with $c=900$ under the hypothesis that $\gg$ is even; see \cite{gk2}. In both cases, the structure of $\aut(\cX)$ was also determined. For $p=2$, an analogous result is not available yet, although a complete classification of the structure of $\aut(\cX)$ was given \cite{gklondon}.

In this paper, the focus is on curves $\cX$ of genus $\gg(\cX)\ge 2$ \textcolor{black}{which satisfy condition}
\begin{itemize}
\item[(**)] \emph{$\cX$ has positive $p$-rank $\gamma(\cX)$.}
\end{itemize}
and the subject is the analog of the Riemann-Hurwitz bound on $|\aut(\cX)|$ where the $p$-rank replaces the genus.
\textcolor{black}{
Unlike the genus, however, the $p$-rank alone does not bound the order of the automorphism
group. The $p$-rank might be insensitive to tame ramification,
and tame covers can enlarge $\aut(\cX)$ arbitrarily while leaving $\gamma(\cX)$ unchanged;
explicit examples are given in Section~\ref{exam}. This is consistent with Serre's
lemma \cite{Serre1960}: $\aut(\cX)$ acts faithfully on the group of geometric $n$-torsion points of the Jacobian of $\cX$
for $n\ge 3$ coprime to $p$, but the coprimality hypothesis cannot be dropped. Indeed, a
faithful action on the geometric $p$-torsion points, which form a group isomorphic to
$(\mathbb{Z}/p\mathbb{Z})^{\gamma(\cX)}$, would give the bound
$|\aut(\cX)|\le|\mathrm{GL}(\gamma(\cX),p)|$. Therefore, any bound in terms of the $p$-rank
must rely on additional assumptions on the automorphism group. In this spirit,  our main result is the following.
}
\begin{theorem}
\label{themain18122025}
Let $\cX$ be a (projective, geometrically irreducible, non-singular) algebraic curve of genus $\ge 2$ and positive $p$-rank $\gamma(\cX)$, defined over an algebraically closed field $\mathbb{K}$ of positive characteristic $p>0$. Let $\Gamma$ be a subgroup of $\aut(\cX)$. Assume that there is a point $P\in \cX$ such that
\begin{itemize}
\item[(*)] the quotient curve $\cX/S_P$ is rational,
\end{itemize}
where $S_P$ denotes the Sylow p-subgroup of the stabilizer $\Gamma_P$ of $P$ in $\Gamma$.
 Then
 the following $p$-rank analog of the Riemann-Hurwitz bound
\begin{equation}
\label{eq18122025}
|\Gamma|\le 24 \left(\frac{p}{p-1}\right)^4 \gamma(\cX)^4
\end{equation}
holds, unless a subgroup of index $\le 2$ of $\Gamma$ fixes $P$, \textcolor{black}{i.e. the $\Gamma$-orbit $\Delta$ of $P$ has size $\le 2$.}
\end{theorem}

Bound (\ref{eq18122025}) is sharp up to a constant depending only on $p$; see Section \ref{ex3} for an example.

\begin{remark}\label{rem:necessity}
\textcolor{black}{Neither (*) nor the \textcolor{black}{hypothesis on $|\Delta|$}  in Theorem~\ref{themain18122025} can be removed, even if  all other conditions are kept. \textcolor{black}{More precisely,}
\begin{itemize}
\item[(i)] \emph{Condition $|\Delta|\ge 3$ is necessary.} The examples in Section \ref{skms} give curves
satisfying (*) and (**) with $|\Delta|=2$ for which $|\Gamma|$ is not bounded by any function
of $\gamma(\cX)$. Since $\Gamma_P$ has index $2$ in $\Gamma$, the same holds for $\Gamma_P$,
whose orbit of $P$ has size $1$;
\item[(ii)] \emph{Condition (*) is necessary.} The examples in Section ~\ref{exfermat} show that, once (*) is
dropped, $|\Gamma|$ admits no upper bound in terms of $\gamma(\cX)$, even though both (**) and
$|\Delta|\ge 3$ hold.
\end{itemize}}
\end{remark}

\textcolor{black}
{\begin{remark}\label{rem:on-star}
Hypothesis $(*)$ concerns the quotient curve $\mathcal{X}/S_P$ only, and it is the
condition under which $p$-subgroups of $\mathrm{Aut}(\mathcal{X})$ are investigated
in \cite{BCP,pries2011}; there, however, the cover $\mathcal{X}\to\mathcal{X}/S_P$ is in
addition unramified outside $P$, which forces
$\gamma(\mathcal{X})=0$. Thus Harbater--Katz--Gabber covers and big actions are
precisely the zero $p$-rank instances of $(*)$, and Theorem \ref{themain18122025} deals
with the positive $p$-rank counterpart of that configuration.
\end{remark}}

\section{Background and basic results from Algebraic Geometry}
\label{prAG}

In this paper, \textcolor{black}{We adopt notation and terminology from \cite{gktrans,gklondon,gk2,HKT,stichtenoth1993}.} In particular,
$\cX$ stands for a (projective, geometrically irreducible, non-singular) algebraic curve of genus $\gg(\cX)\ge 2$ defined over an algebraically closed field $\mathbb{K}$ of positive characteristic $p>0$. Furthermore, \textcolor{black}{$\mathbb{K}(\cX)$ is the function field of $\cX$}, and $\aut(\cX)$ denotes the automorphism group of $\cX$ fixing $\mathbb{K}$ elementwise.

For a subgroup $G$ of $\aut(\cX)$, let $\bar \cX$ denote a non-singular model of $\K(\cX)^G$, that is,
an algebraic
curve with function field $\K(\cX)^G$, where $\K(\cX)^G$ consists of all elements of $\K(\cX)$
fixed by every element in $G$. As customary, $\bar \cX$ is called the
quotient curve of $\cX$ by $G$ and denoted by $\cX/G$. The field extension $\K(\cX)|\K(\cX)^G$ is  Galois of degree $|G|$.

Let $\Phi$ be the Galois cover of $\cX\to \bar{\cX}$ where $\bar{\cX}=\cX/G$ is a quotient curve of $\cX$ with respect to $G$.
 A point $P\in\cX$ is a ramification point of $G$ if the stabilizer $G_P$ of $P$ in $G$ is non-trivial; the ramification index $e_P$ is $|G_P|$; a point $\bar{Q}\in\bar{\cX}$ is a branch point of $G$ if there is a ramification point $P\in \cX$ such that $\Phi(P)=\bar{Q}$; the ramification (branch) locus of $G$ is the set of all ramification (branch) points. The $G$-orbit $o$ of $P\in \cX$ is the subset
$o=\{R\mid R=g(P),\, g\in G\}$, and it is {\em long} if $|o|=|G|$, otherwise $o$ is {\em short}. For a point $\bar{Q}$, the $G$-orbit $o$ lying over $\bar{Q}$ consists of all points $P\in\cX$ such that $\Phi(P)=\bar{Q}$. For $P\in o$, the Orbit theorem states that $|o|=|G|/|G_P|$; hence $\bar{\cQ}$ is a branch point if and only if $o$ is a short $G$-orbit. It may be that $G$ has no short orbits. This is the case if and only if every non-trivial element in $G$ is fixed--point-free on $\cX$, that is, the cover $\Phi$ is unramified. On the other hand, $G$ has a finite number of short orbits. For a non-negative integer $i$, the $i$-th ramification group of $\cX$
at $P$ is denoted by $G_P^{(i)}$ (or $G_i(P)$ as in \cite[Chapter
IV]{serre1979})  and defined to be
$$G_P^{(i)}=\{g\mid \ord_P(g(t)-t)\geq i+1, g\in
G_P\}, $$ where $t$ is a uniformizing element (local parameter) at
$P$. Here $G_P^{(0)}=G_P$.
The structure of $G_P$ is well known; see for instance \cite[Chapter IV, Corollary 4]{serre1979} or \cite[Theorem 11.49]{HKT}:
\begin{result}
\label{res74} The stabilizer $G_P$ of a point $P\in \cX$ in $G$ has the following properties.
\begin{itemize}
\item[\rm(i)] $G_P^{(1)}$ is the unique Sylow $p$-subgroup of $G_P$;
\item[\rm(ii)] For $i\ge 1$, $G_P^{(i)}$ is a normal subgroup of $G_P$ and the quotient group $G_P^{(i)}/G_P^{(i+1)}$ is an elementary abelian $p$-group.
\item[\rm(iii)] $G_P=G_P^{(1)}\rtimes Z$ where the complement $Z$ is a cyclic group whose order is prime to $p$.
\end{itemize}
\end{result}
The Hurwitz genus formula  is the following equation
    \begin{equation}
    \label{eq1}
2\gg(\cX)-2=|G|(2\gg(\bar{\cX})-2)+\sum_{P\in \cX} d_P.
    \end{equation}
    where
\begin{equation}
\label{eq1bis}
d_P= \sum_{i\geq 0}(|G_P^{(i)}|-1).
\end{equation}
Here $D(\cX|\bar{\cX})=\sum_{P\in\cX}d_P$ is the Hilbert {\emph{different}}. For a tame subgroup $G$ of $\aut(\cX)$, that is for $p\nmid |G_P|$,
$$\sum_{P\in \cX} d_P=\sum_{i=1}^m (|G|-\ell_i)$$
where $\ell_1,\ldots,\ell_m$ are the sizes of the short orbits of $G$.

A subgroup of $\aut(\cX)$ is a $p'$-group (or a prime to $p$ group) if its order is prime to $p$. A subgroup $G$ of $\aut(\cX)$ is {\em{tame}} if the $1$-point stabilizer of any point in $G$ is a $p'$-group. Otherwise, $G$ is {\em{non-tame}} (or {\em{wild}}). Obviously, every $p'$-subgroup of $\aut(\cX)$ is tame, but the converse is not always true. From the classical Hurwitz's bound,
if $|G|>84(\gg(\cX)-1)$ then $G$ is non-tame; see  \cite{stichtenoth1973II} or \cite[Theorem 11.56]{HKT}.
An orbit $o$ of $G$ is {\em{tame}} if $G_P$ is a $p'$-group for $P\in o$, otherwise $o$ is a {\em{non-tame orbit}} of $G$. The following result due to Stichtenoth is on the number of short orbits of large automorphism groups \cite{stichtenoth1973II}; see also Theorems 11.56 and 11.116 in  \cite{HKT}.
\begin{result}
\label{res56.116}
If the order of $G$ exceeds $84(\gg(\cX)-1)$, then $G$ has either
\begin{itemize}
\item[\rm(a)] exactly three short orbits$,$ two tame and
one non-tame$,$ and $|G|< 24\gg(\cX)^2$, or
\item[\rm(b)] exactly two short orbits$,$ both non-tame$,$ and  $|G|<16 \gg(\cX)^2$, or
\item[\rm(c)] only one short orbit which is non-tame, and its size divides $2\gg(\cX)-2$, or
\item[\rm(d)] exactly two short orbits$,$ one tame and one non-tame.
\end{itemize}
\end{result}

The $p$-rank of $\cX$ is  defined to be the rank of the (elementary abelian) group of the $p$-torsion points in the Jacobian variety of $\cX$, and it coincides with the Hasse-Witt invariant of $\cX$. The $p$-rank of $\cX$ is denoted by $\gamma(\cX)$.  Then $0\le \gamma(\cX)\le \gg(\cX)$, and $\cX$ is an \emph{ordinary curve} if $\gamma(\cX)=\gg(\cX)$.
Let $\bar{\cX}=\cX/S$ for a $p$-subgroup $S$ of $\aut(\cX)$. The Deuring-Shafarevich formula, see \cite{sullivan1975} or \cite[Theorem 11.62]{HKT}, states that

\begin{equation}
    \label{eq2deuring}
\gamma(\cX)-1={|S|}(\gamma({\bar{\cX}})-1)+\sum_{i=1}^k (|S|-\ell_i)
    \end{equation}
where $\ell_1,\ldots,\ell_k$ are the sizes of the short orbits of $S$.

The Stichtenoth bound, see \cite{stichtenoth1973I,stichtenoth1973II}, is $|\aut(\cX)|\le 16\mathfrak{g}(\cX)^4$ valid for every curve $\cX$ of genus $\mathfrak{g}(\cX)\ge 2$, with the unique exception of
the Hermitian curve of genus $\ha q(q-1)$ whose automorphism group has order $(q^3+1)q^3(q^2-1)$ where $q$ is any power of the characteristic. An improvement on the Stichtenoth bound was given by Henn (\cite{henn1978}; see also \cite[Theorem 11.127]{HKT}) who was able to determine the curves $\cX$ with $|\aut(\cX)|\geq 8\gg(\cX)^3$.

For automorphism groups with special structures or actions on $\cX$, the classical Hurwitz bound can be improved; see for instance \cite[Theorems 11.108, 11.79, 11.60]{HKT}, and \cite{DiGi,KMo,KLT}.
\begin{result}
\label{res60.79.108} Let $\mathfrak{g}(\cX)\ge 2$.
\begin{enumerate}
\item[\rm(i)] If $G$ is abelian then $|G|\leq 4\gg(\cX)+4$.
\item[\rm(ii)] If $G$ fixes a point $P$ of $\cX$ and its order is prime to $p$ then $|G|\leq 4\gg(\cX)+2$.
\end{enumerate}
\end{result}
\begin{result}
\label{resratcurve} If $\cX$ is a rational curve, then  $\aut(\cX)\cong PGL(2,\mathbb{K})$.
\end{result}

\begin{result} (Nakajima \cite{nakajima1987}, \cite[Theorem 11.84]{HKT})
\label{resnaka} Let $S$ be a $p$-subgroup of $\aut(\cX)$. If $\cX$ has $p$-rank $\gamma\ge 2$, then
$$|S|\le
\begin{cases}
\frac{p}{p-2}(\gamma-1)\,\mbox{when $p\ge 3$};\\
4(\gamma-1)\,\,\, \mbox{when $p=2$};\\
\end{cases}
$$
If $\cX$ has $p$-rank $\gamma=1$, then
$$
\begin{cases}
{\mbox{$|S|$ divides $(\gg(\cX)-1)$, $S$ is cyclic and it has no fixed point when $p\ge 3$}};\\
|S|\le 4(\gg(\cX)-1)\,\,\, \mbox{when $p=2$}.
\end{cases}
$$
\end{result}

\begin{result} (\cite[Lemma 3.2]{gklondon},\cite[Lemma 11.131]{HKT})
\label{lem11.131}
 Let $S$ be a $p$-subgroup of $\aut(\cX)$ fixing a point $P\in\cX$. If
 the quotient curve $\cX/S$ has p-rank zero, and no non-trivial element of $S$ fixes a point other than $P$, then $\cX$ has $p$-rank zero, as well.
\end{result}

\begin{result}[Dickson's classification of finite subgroups of the projective linear group $PGL(2,\mathbb{K})$]
\label{resdickson}
Any finite subgroup of the group $PGL(2,\mathbb{K})$ is isomorphic to
a subgroup of $\PGL(2,q)$ for some $q=p^h$. The subgroups of $\PGL(2,q)$
are listed below.
\begin{itemize}
\item[\rm(i)] cyclic group whose order divides $q\pm1$;
\item[\rm(ii)] elementary abelian $p$-group whose order divides  $q$;
\item[\rm(iii)] dihedral groups whose order divides $2(q\pm 1)$;
\item[\rm(iv)] the alternating group ${\rm{Alt}}_4$, for $p > 2$, or $p = 2$ and $h$ even;
\item[\rm(v)] the symmetric group ${\rm{Sym}}_4$, for $p>2$;
\item[\rm(vi)] the alternating group ${\rm{Alt_5}}$, $q^2\equiv 1 \pmod{5} $;
\item[\rm(vii)] the semidirect product of an elementary abelian
$p$-group of order $p^a$, $a\le h$, by a cyclic group of order $n>1$ with $n\mid(p^a-1);$
\item[\rm(viii)] $\PSL(2,p^f)$, $f\mid h$;
\item[\rm(ix)] $\PGL(2,p^f)$, $f\mid h$.
\end{itemize}
In particular, no non-trivial element of $\PGL(2,\mathbb{K})$ fixes more than two points.
\end{result}

\section{General results under hypotheses (*) and (**)}
\label{gr5}

We prove some further preliminary results concerning curves $\cX$ of genus $\gg(\cX)\ge 2$ which satisfy hypotheses (*) and (**). Throughout the section,  $\Gamma\le \aut(\cX)$ stands for a non-tame automorphism group together with a $\Gamma$-orbit $\Delta$ and a point $P\in \Delta$ such that the Sylow $p$-subgroup $S_P$ of $\Gamma_P$ is not trivial. We also assume  $|\Delta|>2$.

\begin{lemma}
\label{lem25jan2022}
Let $S$ be a $p$-subgroup of $\aut(\cX)$ and $\bar{\cX}=\cX/S$ be the quotient curve of $\cX$ with respect to $S$.
Assume that the cover $\cX|\bar{\cX}$ ramifies, and let $p^\ell$ be the size of the longest short orbit of $S$, and $p^s$ be the size of the shortest orbit of $S$.
Then the number of points of $\cX$ which are fixed by at least one non-trivial element of $S$ is at most
$\frac{p^\ell}{|S|-p^\ell}\,(\gamma(\cX)+p^s-1)+p^s$.
\end{lemma}
\begin{proof} Let $\Omega_1,\ldots,\Omega_m$ denote the short orbits of $S$, $q=|S|$, and  $N=\sum_{i=1}^m |\Omega_i|$. From the Deuring-Shafarevich formula applied to the $p$-ranks $\gamma(\cX)$ and $\gamma(\bar{\cX})$,
\begin{equation}\label{DSLemma}
\gamma(\cX)-1=|S|(\gamma(\bar{\cX})-1)+\sum_{i=1}^m (|S|-|\Omega_i|).
\end{equation}

If $\gamma(\bar{\cX})\ge 1$ then
$$\gamma(\cX)-1\ge \sum_{i=1}^m (|S|-|\Omega_i|)=\sum_{i=1}^{m}\left(\frac{|S|}{|\Omega_i|}-1\right)|\Omega_i|\ge \left(\frac{q}{p^\ell}-1\right)\sum_{i=1}^{m}|\Omega_i|=\left(\frac{q}{p^\ell}-1\right)N$$ whence the claim follows.

From now on $\gamma(\bar{\cX})=0$ is assumed. Let $\Omega_m$ denote the shortest orbit of $S$. Then $$\sum_{i=1}^m (|S|-|\Omega_i|)=|S|-p^s+\sum_{i=1}^{m-1} (|S|-|\Omega_i|),$$ which combined with \eqref{DSLemma} yields
$$\gamma(\cX)-1=-p^s+\sum_{i=1}^{m-1} (|S|-|\Omega_i|)=-p^s+\sum_{i=1}^{m-1}\left(\frac{|S|}{|\Omega_i|}-1\right)|\Omega_i|\ge -p^s+\left(\frac{q}{p^\ell}-1\right)(N-p^s).$$
 whence the  claim follows.
\end{proof}

Let
$\Delta_0$ denote the set of all points of $\cX$ other than $P$ which are fixed by some non-trivial element of $S_P$.

\begin{lemma}
\label{lem24072024} Assume that (*) and (**) hold. Then $S_P$ has a short orbit other than $\{P\}$. Moreover, if $S_P$ has a unique short orbit $\Delta_0$ other than $\{P\}$, then no non-trivial element of $\Gamma_P$ fixes a point of $\cX$ outside $\{P\}\cup \Delta_0$, and hence $\Delta_0$ is also the unique short orbit of $\Gamma_P$ other than $\{P\}$.
\end{lemma}
\begin{proof} The first claim follows from Result \ref{lem11.131}.
By (iii) of Result \ref{res74}, $S_P$ is a normal subgroup of $\Gamma_P$. Since $\Delta_0$ is the unique short orbit of $S_P$ other than $\{P\}$, $\Gamma_P$ preserves $\Delta_0$.
Let $\bar{\cX}=\cX/S_P$ be the quotient curve of $\cX$ by $S_P$. From (iii) of Result \ref{res74}, $\bar{\Gamma}_P=\Gamma_P/S_P$ is a cyclic subgroup of $\aut(\bar{\cX})$ whose order is prime to $p$, and $\bar{\Gamma}_P$ fixes the point $\bar{R}$ lying under $\Delta_0$.
If $\Gamma_P$ coincides with $S_P$ the statement clearly holds. Assume then that $\Gamma_P\ne S_P$.
Result \ref{resdickson} together with (*) yield that $\bar{\Gamma}_P$ fixes exactly two points of $\bar{\cX}$, one of them is the point $\bar{P}$ lying under $P$ the other is $\bar{R}$.
On the other hand, let $Q\in \cX$ be a point fixed by some non-trivial element of $\Gamma_P$. Then $H=\Gamma_{P,Q}$ is non-trivial. If $Q\not\in \Delta_0$ then $H$ has order prime to $p$.
Therefore, $HS_P/S_P$ is a subgroup of $\bar{\Gamma}_P$ fixing the point $\bar{Q}$ of $\bar{\cX}$ lying under $Q$. It turns out that $\bar{Q}=\bar{R}$, and hence $Q\in \Delta_0$.
\end{proof}

\begin{lemma}
\label{leterribileHW} If both (*) and (**) hold, then $$\gamma(\cX)\ge \ha |\Gamma_P|,$$ unless $S_P$  has a unique short orbit other than $\{P\}$.
Moreover, if $\Gamma_P\ne S_P$, then $\Gamma_P$ preserves only one (not necessarily short) $S_P$-orbit other than $\{P\}$.
\end{lemma}
\begin{proof}
Let $\{P\}, \Omega_1,\ldots, \Omega_n$ denote the short orbits of $S_P$. In the cover $\cX|\bar{\cX}$ with $\bar{\cX}=\cX/S_P$, let $\bar{P},\bar{Q_1},\ldots,\bar{Q}_n$ denote the points of $\bar{\cX}=\cX/S_P$ lying under $\{P\}, \Omega_1,\ldots, \Omega_n$, respectively. From Result \ref{res74}, $\Gamma_P=S_P\rtimes H$ with a cyclic complement $H$ of order $k$ with $p\nmid k$. Therefore,  $\bar{H}=\Gamma_P/S_P\cong H$ is a subgroup of $\aut(\bar{\cX})$.
If $k=1$, that is $\Gamma_P=S_P$, then the Deuring-Shafarevich formula applied to $S_P$ yields
$$
\gamma(\cX)-1=-|S_P|+|S_P|-1+\sum_{i=1}^n (|S_P|-|\Omega_i|)\geq -1+\ha n|S_P|,
$$
whence $\gamma(\cX)\ge \ha n|S_P|   \ge \ha |\Gamma_P|$.
Therefore $k>1$ may be assumed. Then $\bar{H}$ is a non-trivial (cyclic) subgroup of $\aut(\bar{\cX})$  of order $k$ of the rational curve $\bar{\cX}$. So, Result \ref{resdickson} implies that $\bar{H}$ has exactly two fixed points in $\bar{\cX}$ and that no non-trivial element of $\bar{H}$ fixes a further point. This proves the last claim. Moreover, one of the fixed points of $\bar{H}$ is $\bar{P}$, the other is a further point $\bar{R}$ not necessarily in $\mathcal{\bar{Q}}=\{\bar{Q}_1,\ldots,\bar{Q}_n\}$. Up to relabeling, either $\bar{R}\notin \mathcal{\bar{Q}}$, or $\bar{R}=\bar{Q}_1$. Assume that $n>1$. Then, in the latter case, we can consider the $\bar{H}$-orbit $\Psi$ of $\bar{Q}_2$, which has length $k$. Now take any point $\bar{T}\in \Psi$.
 Then $\bar{T}$ is mapped by any element in $\bar{H}$ to a point in $\mathcal{\bar{Q}}$ as $H$ takes any short orbit of $S_P$ to a short orbit of $S_P$.
Therefore, $n\ge k +1$. A similar argument shows that if $\bar{R}\notin \mathcal{\bar{Q}}$ then $n\ge k$. From the Deuring-Shafarevich formula applied to $S_P$, we have
$$\gamma(\cX)-1=-|S_P|+|S_P|-1+\sum_{i=1}^n (|S_P|-|\Omega_i|).$$ Since $|S_P|-|\Omega_i|\ge \ha |S_P|$,
we obtain $\gamma(\cX)\ge \ha k |S_P|=\ha |\Gamma_P|$.
\end{proof}

\begin{lemma}
\label{lem27092025A} Assume that both (*) and (**) hold. If  $S_P$ has a unique short orbit $\Delta_0$  other than $\{P\}$,  then
$$\ha |S_P|\le \gamma(\cX)\le |S_P|-1.$$

\end{lemma}
\begin{proof}
The claim is a straightforward consequence of the Deuring-Shafarevich formula applied to $S_P$.
\end{proof}

Take a point $Q\in \Delta\setminus \{P\}$ and let $W=S_P\cap S_Q$. Observe that $S_P$ and $S_Q$ are conjugate in $\Gamma$, hence (*) ensures that both $\cX/S_P$ and $\cX/S_Q$ are rational curves.
\begin{lemma}
\label{lem14set23} Assume that both (*) and (**) hold. Take a point $Q\in \Delta\setminus \{P\}$ and let $W=S_P\cap S_Q$.
Let $x,y\in \mathbb{K}(\cX)$ such that $\mathbb{K}(x)$ and $\mathbb{K}(y)$ are the subfields of $\mathbb{K}(\cX)$ fixed element-wise by $S_P$ and $S_Q$ respectively.
Then
$\mathbb{K}(x,y)$ is the fixed field of $W=S_P \cap S_Q$ in $\mathbb{K}(\cX)$. Moreover, if $\bar{\mathfrak{g}}$ is the genus of $\mathbb{K}(x,y)$, we have $$\bar{\mathfrak{g}}\leq \left(\frac{|S_P|}{|W|}-1 \right)^2$$
\end{lemma}
\begin{proof}
Observe that the fixed field of $\langle S_P,S_Q \rangle$ in $\mathbb{K}(\cX)$ is $\mathbb{K}(x)\cap \mathbb{K}(y)$.
By definition, $W$ fixes both $x$ and $y$, and so the fixed field of $W$ contains $\mathbb{K}(x,y)$. On the other hand, from Galois theory,
there exists a subgroup $U$ of $\langle S_P,S_Q \rangle$ whose fixed field is $\mathbb{K}(x,y)$. Observe that since
$\mathbb{K}(x,y)$ is an intermediate field of both the extensions $\mathbb K(\cX):\mathbb K(x)$
and $\mathbb K(\cX):\mathbb K(y)$, $U$ is contained in both $S_P$ and $S_Q$, and hence in $W$. Actually, $W=U$ since $W$ fixes both $x$ and $y$. This proves
the first claim.
The Castelnuovo inequality, see \cite[Theorem 3.11.3]{stichtenoth1993}, yields the second claim.
\end{proof}

We now define a further condition:
\begin{itemize}
\item[(***)]
every point of the orbit $\Delta$ of $P$ under the action of $\Gamma $ is fixed by some non-trivial element of $S_P$, i.e.
$\Delta \subseteq \{P\}\cup \bigcup_{i=1}^n \Omega_i$, where $\Omega_1,\ldots,\Omega_n$ are the short orbits of $S_P$ other than $\{P\}$.
\end{itemize}
Note that by the definition of $\Delta_0$, (***) can be read as $\Delta \subseteq   \{P\}\cup \Delta_0$.
The following Sections will distinguish the cases where (***) holds and where it does not.

\section{Case (A): $\Delta \subseteq \{P\}\cup \bigcup_{i=1}^n \Omega_i$}
\label{pr***}
Throughout this section, besides (*), (**), and $|\Delta|>2$ we assume that (***) holds.
Lemma \ref{lem25jan2022} with $p^s=1$ provides the following bound.
\begin{lemma}
\label{leA260223} If both {\rm{(*)}} and {\rm{(***)}} hold and $|\Delta|>2$, then
\begin{equation}
\label{eqB260223}
|\Delta|\le \frac{p^\ell}{|S_P|-p^\ell}\,\gamma(\cX)+1 \le \frac{1}{p-1}\gamma(\cX)+1,
\end{equation}
and $\gamma(\cX)\ge 2(p-1)$.
\end{lemma}

In particular, by Lemma \ref{leA260223} and $|\Delta|>2$, condition  (***) implies (**).
\begin{proposition}
\label{pro16072025}
Assume that both {\rm{(*)}} and {\rm{(***)}} hold. If $\mathfrak{g}(\cX)\ge 2$ and $|\Delta|>2$,  then
\begin{equation}
\label{eqpr}
|\Gamma|\le
2\left(\frac{1}{p-1} \gamma(\cX)+1\right)\gamma(\cX)
\end{equation}
unless $S_P$ is transitive on $\Delta\setminus\{P\}$.
\end{proposition}
\begin{proof} If $S_P$ is not transitive on $\Delta\setminus \{P\}$, then (***) yields that $S_P$ has at least two short orbits on $\Delta$ other than $\{P\}$. From Lemma \ref{leterribileHW}, $\gamma(\cX)\ge \ha\, |\Gamma_P|$. Now, the claim follows from (\ref{eqB260223}) and the Orbit theorem.
\end{proof}

\begin{proposition}\label{11set26}
Assume that $\mathfrak{g}(\cX)\ge 2$ and that  both (*) and (***) hold. Let $\Gamma_P=S_P\rtimes H$. If  $|\Delta|>2$  and
$S_P$ is transitive on $\Delta\setminus\{P\}$, then each of the following holds:
\begin{itemize}
\item[($\mathfrak{I}$)] $|H|\le 4 |\Delta|^2$;
\item[($\mathfrak{II}$)] $|\Gamma|\le 8 \gamma(\cX)(\frac{\gamma(\cX)}{p-1}+1)^3\le 24\left(\frac{p}{p-1}\right)^4\gamma(\cX)^4$.
\end{itemize}
\end{proposition}
\begin{proof}
Clearly ($\mathfrak{I}$) holds for $|H|=1$. Then we assume that $H$ is not trivial.
Choose a point $Q$ in $\Delta$, $P\ne Q$ and let $W$ as in Lemma \ref{lem14set23}.
Write $m=|\Delta|-1$.
Note that $W$ is the Sylow $p$-subgroup of $\Gamma_P\cap \Gamma_Q$; therefore, there exists a cyclic subgroup $C$ such that  $\Gamma_P\cap \Gamma_Q$ is the semidirect product of $W$ and $C$.
We remark that the size of $C$ is equal to the size of $H$; this follows by the orbit-stabilizer theorem applied to both $S_P$ and $\Gamma_P$, whose actions are transitive on $\Delta\setminus \{P\}$; hence, both
$$
|S_P\cap S_Q|\cdot m=|S_P|, \qquad |\Gamma_P\cap \Gamma_Q|\cdot m=|\Gamma_P|
$$
hold, that is,
$$
|W|\cdot m=|S_P|, \qquad |W|\cdot |C|\cdot m=|S_P|\cdot |H|
$$
whence $|C|=|H|$.
Now, consider the action of the factor group $(\Gamma_P\cap \Gamma_Q)/W\cong H$ on the curve $\mathcal X/W$. Taking into account $|S_P|=m|W|$, by Lemma \ref{lem14set23} the genus ${\bar{\mathfrak{g}}}$ of
$\mathcal X/W$ is at most $(m-1)^2$.
We distinguish three cases:
\begin{itemize}
\item ${\bar{\mathfrak{g}}}\ge 2$. Here, as $(\Gamma_P\cap \Gamma_Q)/W$ is cyclic, by Result
\ref{res60.79.108}
$|H|\le 4({\bar{\mathfrak{g}}}+1)$, hence
$$
|H|\le 4((m-1)^2+1)\le 4|\Delta|^2
$$
and ($\mathfrak{I}$) holds.
\item ${\bar{\mathfrak{g}}}=1$. By the Riemann-Hurwitz genus formula applied to the covering $\mathcal X/W$ over $\mathcal X/(\Gamma_P\cap \Gamma_Q)$,
$$
0=|H|(2\mathfrak{g}(\mathcal X/(\Gamma_P\cap \Gamma_Q))-2)+2(|H|-1)+\sum_{i=1}^k(|H|-|\Theta_i|)
$$
where $\Theta_1,\ldots,\Theta_k$ are  the short orbits of $(\Gamma_P\cap \Gamma_Q)/W$  different from the points $\bar P$ and $\bar Q$ lying under $P$ and $Q$ respectively.
Necessarily,
$\mathfrak{g}(\mathcal X/(\Gamma_P\cap \Gamma_Q))=0$ holds. As $|\Theta_i|\le |H|/2$ we immediately get $|H|\le 4$ and
($\mathfrak{I}$) holds.
\item ${\bar{\mathfrak{g}}}=0$. Here a cyclic group $\bar C$ of order $|H|$ acts on the projective line fixing $\bar P$ and $\bar Q$. Then it must act semiregularly on the other points of $\mathcal X/W$. Since $|\Delta|>2$ and every element in $\Gamma$ preserves $\Delta$, we have that an orbit of $\bar C$ must be contained in the points lying under $\Delta \setminus \{P,Q\}$. Then
$$
|H|\le |\Delta|-2
$$
and ($\mathfrak{I}$) holds.
\end{itemize}
By Lemma \ref{leterribileHW} or Lemma \ref{lem27092025A} we have $|S_P|\le 2\gamma(\cX)$. Then
to prove ($\mathfrak{II}$) it is enough
to use the Orbit theorem together with Lemma \ref{leA260223}:
$$
|\Gamma|=|S_P|\cdot |H| \cdot |\Delta|\le 2 \gamma(\cX) \cdot 4  \cdot \left(\frac{\gamma(\cX)}{p-1}+1\right)^3.
$$
It remains to observe that for $\gamma(\cX)\ge 1$ and $p\ge 2$
$$
 8 \gamma(\cX)   \cdot \left(\frac{\gamma(\cX)}{p-1}+1\right)^3\le 24\left(\frac{p}{p-1}\right)^4\gamma(\cX)^4,
 $$
 which is
a straightforward computation.
\end{proof}

\section{Case (B):  $\Delta\not\subseteq \{P\}\cup \bigcup_{i=1}^n \Omega_i $}
\label{pro***no}

Throughout this section, the curve $\cX$ is equipped with a subgroup $\Gamma$ of $\aut(\cX)$ and a point $P\in\cX$ such that both (*) and (**) hold. We investigate the case where (***) does not hold. This time the behavior of $\Delta$ changes as it contains a point $Q$ such that $S_P\cap S_Q$ is trivial. {\color{black} This case is not empty: two families of curves satisfying (*), (**) and $|\Delta|>2$ but not (***) are exhibited in Sections \ref{caseB} and \ref{exzieve}.}
From the last claim of Lemma \ref{lem24072024}, if $S_P$ has exactly one short orbit other than $\{P\}$, then this holds true for $\Gamma_P$.
We will use several times that $\cX/\Gamma$ is a rational curve, as a consequence of (*) together with L\"uroth theorem (see e.g. \cite[Proposition 3.5.9]{stichtenoth1993}).

As in the previous Sections, $\Delta_0$ denotes the (non-empty) set of all points of $\cX$ other than $P$ which are fixed by some non-trivial element of $S_P$. Equivalently, for the quotient curve $\bar{\cX}=\cX/S_P$,  $\{P\}\cup \Delta_0$ is the set of all points in $\cX$ where the cover $\cX|\bar{\cX}$ ramifies.
If (***) does not hold, then $\Delta$ is not contained in $\{P\}\cup \Delta_0$. For any point $Q\in \Delta\setminus (\{P\}\cup \Delta_0)$, $S_P$ and $S_Q$ have trivial intersection. Let $x$ and $y$ be a generator of the fixed field of $S_P$ and $S_Q$, respectively.
Let
\begin{equation}
\label{eq01092025}
\mbox{$v=d_P(\Gamma)-|\Gamma_P|=-1+|S_P^{(1)}|-1+\ldots+|S_P^{(m)}|-1$, with $|S_P^{(m+1)}|=1.$}
\end{equation}
Note that since $S_P^{(1)}=S_P$ we have
\begin{equation}\label{eq:I1}
v\ge |S_P|-2.
\end{equation}

We will apply Lemma 11.82 in \cite{HKT}, which we reformulate here for $G=\Gamma_P$.
Recall that, for a local parameter $z$ at $P$,
$$
\Gamma_P^{(i)}=\{\alpha \in \Gamma_P\mid \ord_P(\alpha(z)-z)\ge i+1\},
$$
and that $\ord_P(\alpha(z)-z)$ does not depend on the choice of $z$.

\begin{lemma}\label{hkt82}
Let $v$ be as in \eqref{eq01092025}. Let
$$\nu_{-1}=0<\nu_0<\nu_1<\ldots$$
be the strictly increasing sequence of all non-negative integers which are of the form
$\ord_P(\alpha(z)-z)$ for some non-trivial $\alpha \in \Gamma_P$, that is, the
\emph{ramification numbers} of $\Gamma_P$ at $P$. For $i\ge 0$, let
$$
q_i=\frac{|\Gamma_P^{(\nu_i-1)}|}{|\Gamma_P^{(\nu_i)}|},
$$
that is, $q_i$ is the index of the $i$-th step of the chain obtained from
$\Gamma_P^{(0)}\ge \Gamma_P^{(1)}\ge\ldots$ by deleting repetitions. Then
 \begin{itemize}
 \item[\rm{(i)}] $|H|$ divides $v+1$;
 \item[\rm{(ii)}] $|H|$ divides $(\nu_i-1)(q_i-1)$.
 \end{itemize}
\end{lemma}

\begin{lemma}\label{lem:one}
 If
$|S_P^{(2)}|< |S_P|$, then
\begin{equation}\label{eq:I2b}
h \ \Big|\ \frac{|S_P|}{|S_P^{(2)}|}-1 .
\end{equation}
\end{lemma}
\begin{proof}
Note that $S_P^{(2)}=\Gamma_P^{(2)}$ and $S_P=\Gamma_P^{(1)}$. If $h=1$ there is nothing
to prove, so assume $h>1$. Then every non-trivial element of $H$ satisfies
$\ord_P(\alpha(z)-z)=1$, whence $\nu_0=1$. Moreover $|S_P^{(2)}|<|S_P|$ means that
$\nu_1=2$, so that
$$
q_1=\frac{|\Gamma_P^{(\nu_1-1)}|}{|\Gamma_P^{(\nu_1)}|}
=\frac{|\Gamma_P^{(1)}|}{|\Gamma_P^{(2)}|}=\frac{|S_P|}{|S_P^{(2)}|}.
$$
Hence the claim follows from Lemma \ref{hkt82}(ii) applied with $i=1$.
\end{proof}

Let $v$ be as in \eqref{eq01092025} and note that
$d_{P}(S_P)=|S_P|+v$.
\begin{lemma}\label{lem:two}
We have
\begin{equation}
h \ \big|\ v+1 .
\end{equation}
\end{lemma}
\begin{proof}
This is Lemma \ref{hkt82}(i).
\end{proof}

The following claim is a consequence of Lemma \ref{lem14set23}.
\begin{lemma}
\label{lem14dic21} Assume that both {\rm{(*)}} and {\rm{(**)}} hold \textcolor{black}{but {\rm{(***)}} does not} hold. Then $\mathbb{K}(\cX)=\mathbb{K}(x,y)$, $\mathfrak{g}(\cX)\le (|S_P|-1)^2$ and $|S_P|\leq \frac{p}{p-1}\gamma(\cX)$. In particular,
\begin{equation}
\label{eq09112023}
\mathfrak{g}(\cX)\le \left(\frac{p}{p-1}\right)^2\gamma(\cX)^2.
\end{equation}
\end{lemma}
\begin{proof}
The claims $\mathbb{K}(\cX)=\mathbb{K}(x,y)$ and $\mathfrak{g}(\cX)\le (|S_P|-1)^2$ follow from Lemma \ref{lem14set23}.
Moreover, from the Deuring-Shafarevich formula applied to $S_P$, we obtain $\gamma(\cX)\ge |S_P|-\frac{1}{p}|S_P|$ which yields $\mathfrak{g}(\cX)< (\frac{p}{p-1})^2\gamma(\cX)^2$.
\end{proof}

We point out some particular cases in which the bound on the order of $\Gamma$ can immediately be obtained by using previous results.

\begin{proposition}
\label{prop290723Ater}  Assume that both {\rm{(*)}} and {\rm{(**)}} hold but \textcolor{black}{{\rm{(***)}} does not} hold.
    If $|\Gamma|\le 84(\mathfrak g(\cX)-1)$
    then
    $|\Gamma|\le 24\left(\frac{p}{p-1}\right)^4\gamma(\cX)^4$.
\end{proposition}
\begin{proof}
Clearly $|\Gamma|\le 84 \left[\left(\frac{p}{p-1}\right)^2\gamma(\cX)^2-1 \right]$
by \eqref{eq09112023}.
The claim then follows from
$$84 \left[\left(\frac{p}{p-1}\right)^2\gamma(\cX)^2-1 \right]\le 24\left(\frac{p}{p-1}\right)^4\gamma(\cX)^4$$
for each positive value of $\gamma(\cX)$.
\end{proof}

\begin{proposition}
\label{prop290723A}  Assume that both {\rm{(*)}} and {\rm{(**)}} hold but \textcolor{black}{{\rm{(***)}} does not} hold.
    If $\Gamma$ has at least two non-tame short orbits, then
    $|\Gamma|\le 24\left(\frac{p}{p-1}\right)^4\gamma(\cX)^4$.
\end{proposition}
\begin{proof}
 The claim is a corollary of (\ref{eq09112023}), Proposition \ref{prop290723Ater}, and (b) of Result \ref{res56.116}.
\end{proof}

\begin{proposition}
\label{prop290723Aprimo}  Assume that both {\rm{(*)}} and {\rm{(**)}} hold but \textcolor{black}{{\rm{(***)}} does not} hold.
    If $\Gamma$ has at least two tame short orbits, then
    $|\Gamma|\le 24\left(\frac{p}{p-1}\right)^4\gamma(\cX)^4$.
\end{proposition}
\begin{proof}
 The claim is a corollary of (\ref{eq09112023}), Proposition \ref{prop290723Ater}, and (a) of Result \ref{res56.116}.
\end{proof}

{\color{black}
\begin{proposition}
\label{pro24072024} Assume that  {\rm{(*)}} and {\rm{(**)}} hold but {\rm{(***)}} does not hold. If $\Gamma$ has a unique short orbit, then
$$|\Gamma|< 4|S_P|^3\le  4\left(\frac{p}{p-1}\right)^3\gamma(\cX)^3.$$
\end{proposition}
\begin{proof}
Write $\Gamma_P=S_P\rtimes H$ with $H$ cyclic of order $h$ prime to $p$ (Result \ref{res74}), and put $s=|S_P|$ and $L=2\gg(\cX)-2$.
The unique short orbit of $\Gamma$ is $\Delta$, and $d_P(\Gamma)=|\Gamma_P|+v=sh+v$ by (\ref{eq01092025}). Since $\cX/\Gamma$ is rational, the Hurwitz genus formula applied to $\Gamma$ gives
$$L=-2|\Gamma|+|\Delta|(sh+v)=|\Delta|(v-sh),$$
because $|\Gamma|=|\Delta|sh$. As $L>0$, the integer $D=v-sh$ is positive, and $|\Gamma|=|\Delta|sh=Lsh/D$. By Lemma \ref{lem:two}, $h$ divides $v+1=(D+1)+sh$, whence $h$ divides $D+1$. If $h\ge 2$ this yields $D+1\ge h$, that is $D\ge h-1\ge \ha h$, while $h/D\le 1$ for $h=1$. In both cases $h/D\le 2$, and Lemma \ref{lem14dic21} gives $L\le 2s(s-2)$, whence
$$|\Gamma|=\frac{Lsh}{D}\le 2Ls\le 4s^2(s-2)<4s^3.$$
The last inequality of the statement follows from $s\le \frac{p}{p-1}\gamma(\cX)$, again by Lemma \ref{lem14dic21}.
\end{proof}
}

\begin{proposition}
\label{prop290723bis} Assume that both {\rm{(*)}} and {\rm{(**)}} hold but {\rm{(***)}} does not hold. If $\Gamma$ has exactly one tame short orbit and one non-tame short orbit,  and $S_P$ has at least two short orbits other than $\{P\}$ then $$|\Gamma|< 24\left(\frac{p}{p-1}\right)^4\gamma(\cX)^4.$$
\end{proposition}
\begin{proof}
Let $\Omega$ be the unique short tame orbit of $\Gamma$.
From the Hurwitz genus formula applied to $\Gamma$, $2\mathfrak{g}(\cX)-2=|\Delta|v-|\Omega|$
where $v$ is defined by (\ref{eq01092025}).
From the Orbit theorem, $|\Delta||\Gamma_P|=|\Omega||\Gamma_Q|$ where $P\in \Delta$ and $Q\in \Omega$. Thus
\begin{equation}
\label{eq08112023}
2\mathfrak{g}(\cX)-2=|\Delta|\left(v-\frac{|\Gamma_P|}{|\Gamma_Q|}\right).
\end{equation}
If $|\Gamma_P|\le \ha |\Gamma_Q|$ then (\ref{eq08112023}) yields $2\mathfrak{g}(\cX)-2\ge |\Delta|(v-\frac{1}{2})\geq |\Delta|(|S_P|-\frac{5}{2})$ whence $2\mathfrak{g}(\cX)-2\ge \ha |\Delta|$ follows, unless $p=2$ and $|S_P|=2$. The latter (sporadic) case cannot actually occur by  Lemma \ref{lem14dic21} as $\mathfrak{g}(\cX)\ge 2$.
This together with (\ref{eq09112023}) and Lemma \ref{leterribileHW} imply $$|\Gamma|=|\Gamma_P||\Delta|\le 2\gamma(\cX)4(\mathfrak{g}(\cX)-1)<8\gamma(\cX)\left(\frac{p}{p-1}\right)^2\gamma(\cX)^2=8\left(\frac{p}{p-1}\right)^2 \gamma(\cX)^3.$$
If $|\Gamma_P|>\ha |\Gamma_Q|$, a different argument is used. Let
$$\sigma=v-\frac{|\Gamma_P|}{|\Gamma_Q|}.$$ Since $2\mathfrak{g}(\cX)-2>0$, Equation (\ref{eq08112023}) shows that $\sigma$ is positive. Moreover, as $v|\Gamma_Q|-|\Gamma_P|$ is an integer, $\sigma|\Gamma_Q|$ is also an integer, hence a positive integer. Therefore $\sigma|\Gamma_Q|\ge 1$. Thus,
$$\sigma\ge \frac{1}{|\Gamma_Q|}>\frac{1}{2|\Gamma_P|}.$$
This together with Lemma \ref{leterribileHW} yield
$$\sigma> \frac{1}{4\gamma(\cX)}$$
whence $|\Delta|\le 4\gamma(\cX)(2\mathfrak{g}(\cX)-2)$. Arguing as before, this implies
$$|\Gamma|=|\Gamma_P||\Delta|<2\gamma(\cX)4\gamma(\cX)2(\mathfrak{g}(\cX)-1)<16 \left(\frac{p}{p-1}\right)^2\gamma(\cX)^4.$$
\end{proof}

\begin{proposition}\label{prop:main}
Assume that both (*) and (**) hold but (***) does not  hold. Assume also that
\begin{itemize}
\item[(a)]
$S_P$ has exactly one short orbit
other than $\{P\}$;
\item[(b)] $\Delta$ is the unique wild short orbit of $\Gamma$;
\item[(c)] $\Gamma$ has exactly one tame short orbit;
\item[(d)] $\mathfrak g(\cX)\le(|S_P|-1)^{2}$; in particular $|S_P|\ge 3$, as $\mathfrak g(\cX)\ge 2$.
\end{itemize}
Then
\[
|\Gamma| \le 8\,|S_P|^{4}\le 8\left(\frac{p}{p-1}\right)^4\gamma(\mathcal X)^4.
\]
\end{proposition}
The proof of Proposition \ref{prop:main} will be given in a series of lemmas.
Recall that $\Gamma_{P}=S_P\rtimes H$, with
$H$ a cyclic subgroup of order $h$ with $(h,p)=1$.

Choose $T$ in the tame short orbit and let $N=v\cdot |\Gamma_T|-|S_P|\cdot h$.
\begin{lemma} We have
$$
2\mathfrak g(\cX)-2=|\Gamma|\cdot \frac{N}{|S_P|\cdot h \cdot |\Gamma_T|}.
$$
\end{lemma}
\begin{proof}
By the Riemann-Hurwitz genus formula,
$$
2\mathfrak g(\cX)-2=-2|\Gamma|+|\Delta|(|S_P|h+v)+\frac{|\Gamma|}{|\Gamma_T|}(|\Gamma_T|-1),
$$
that is
$$
2\mathfrak g(\cX)-2=|\Gamma|\left(\frac{v}{|S_P|\cdot h}-\frac{1}{|\Gamma_T|}\right),
$$
whence the claim.
\end{proof}
\begin{lemma}\label{easystaff} If \begin{equation}\label{eq:target}\frac{h\cdot  |\Gamma_T|}{N}\le\frac{ 4 |S_P|^2}{|S_P|-2},
\end{equation}
then $|\Gamma|\le 8 |S_P|^4$.
\end{lemma}
\begin{proof} By (d) we have
$$
|\Gamma|\cdot \frac{N}{|S_P|\cdot h \cdot |\Gamma_T|}=
2\mathfrak g(\cX)-2\le 2|S_P|(|S_P|-2).
$$
Then the hypothesis yields
$$
\frac{|\Gamma|\cdot (|S_P|-2)}{4|S_P|^3}\le 2|S_P|(|S_P|-2)
$$
and the claim follows.
\end{proof}
Let $\mu\ge 0$ be such that
\begin{equation}\label{eq:a}
v=(\mu+1)h-1
\end{equation}
(see Lemma
\ref{lem:two}). Also, set
$$
\ell=(\mu+1)|\Gamma_T|-|S_P|.
$$
\begin{lemma}\label{lem:identity}
We have $\ell\ge1$, and
\begin{equation}\label{eq:rel}
N=h\ell-|\Gamma_T|,\qquad \ab N=\ell v-|S_P| ,
\end{equation}
\begin{equation}\label{eq:master}
\frac{h|\Gamma_T|}{N}=\frac{(v+1)(|S_P|+\ell)}{\ab^{2}N} .
\end{equation}
Moreover, setting \begin{equation}\label{eq:u}
u:=\ab N=\ell v-|S_P|\ \ge\ 1 ,
\end{equation}
we have
\begin{equation}\label{eq:factor}
\frac{h|\Gamma_T|}{N}=\frac{1}{\ab}\Bigl(1+\frac{|S_P|}{\ell}\Bigr)
\Bigl(1+\frac{|S_P|+\ell}{u}\Bigr) .
\end{equation}
\end{lemma}
\begin{proof}
From \eqref{eq:a}, $|\Gamma_T| v=|\Gamma_T|(\ab h-1)=\ab|\Gamma_T|h-|\Gamma_T|$, so
\[
N=|\Gamma_T| v-|S_P|h=h\bigl(\ab|\Gamma_T|-|S_P|\bigr)-|\Gamma_T|=h\ell-|\Gamma_T| ,
\]
and $N>0$ with $|\Gamma_T|\ge1$ forces $h\ell>|\Gamma_T|\ge1$, whence $\ell\ge1$.
Multiplying by $\mu+1$ and using $\ab|\Gamma_T|=|S_P|+\ell$ and $\ab h=v+1$,
\[
\ab N=\ab|\Gamma_T|\,v-\ab|S_P|h=(|S_P|+\ell)v-|S_P|(v+1)=\ell v-|S_P| .
\]
Since $h=(v+1)/\ab$ and $|\Gamma_T|=(|S_P|+\ell)/\ab$,
\[
\frac{h|\Gamma_T|}{N}=\frac{(v+1)(|S_P|+\ell)}{\ab^{2}N},
\]
and $\ab^{2}N=\ab(\ell v-|S_P|)$, so the right-hand side equals
$$(v+1)(|S_P|+\ell)/\bigl(\ab(\ell v-|S_P|)\bigr).$$
For \eqref{eq:factor}, note that $u=\ell v-|S_P|$ gives
$\ell v=u+|S_P|$, hence
\[
v+1=\frac{u+|S_P|+\ell}{\ell},
\]
and therefore, using $\ab^{2}N=\ab u$,
\begin{align*}
\frac{h|\Gamma_T|}{N}
&=\frac{(u+|S_P|+\ell)(|S_P|+\ell)}{\ab\,\ell\,u}
=\frac{1}{\ab}\cdot\frac{|S_P|+\ell}{\ell}\cdot\frac{u+|S_P|+\ell}{u}\\[2pt]
&=\frac{1}{\ab}\Bigl(1+\frac{|S_P|}{\ell}\Bigr)\Bigl(1+\frac{|S_P|+\ell}{u}\Bigr).
\qedhere
\end{align*}
\end{proof}
\begin{lemma}\label{lem:u}
$u\ \ge\ (\ell-1)(|S_P|-2)-2$.
\end{lemma}
\begin{proof}
By \eqref{eq:I1}, $v\ge|S_P|-2$, so
$u=\ell v-|S_P|\ge\ell(|S_P|-2)-|S_P|=(\ell-1)(|S_P|-2)-2$.
\end{proof}
\begin{lemma}\label{lem:ratio}
Let $|S_P|\ge5$ and $\ell\ge2$. Then $(\ell-1)(|S_P|-2)-2\ge|S_P|-4>0$ and
\[
\frac{|S_P|+\ell}{(\ell-1)(|S_P|-2)-2}\ \le\ \frac{|S_P|+2}{|S_P|-4},
\]
with equality for $\ell=2$.
\end{lemma}
\begin{proof}
The first assertion is clear from $\ell-1\ge1$ and $|S_P|\ge5$. The inequality is
equivalent to
\[
\Psi(\ell):=(|S_P|+2)\bigl[(\ell-1)(|S_P|-2)-2\bigr]-(|S_P|+\ell)(|S_P|-4)\ \ge\ 0 .
\]
Note that
\[
\Psi(2)=(|S_P|+2)\bigl[(|S_P|-2)-2\bigr]-(|S_P|+2)(|S_P|-4)=0
\]
and
\[
\Psi(\ell+1)-\Psi(\ell)=(|S_P|+2)(|S_P|-2)-(|S_P|-4)=|S_P|^{2}-|S_P|>0 .
\]
Hence $\Psi(\ell)\ge\Psi(2)=0$ for every $\ell\ge2$.
\end{proof}
\begin{lemma}\label{lem:lge2}
If $\ell\ge2$ then \eqref{eq:target} holds.
\end{lemma}
\begin{proof}
Let $|S_P|\ge5$. In \eqref{eq:factor} use $\ab\ge1$ for the first factor,
$\ell\ge2$ for the second, and Lemmas~\ref{lem:u} and \ref{lem:ratio} for the
third:
\begin{align*}
\frac{h|\Gamma_T|}{N}
&\ \le\ \Bigl(1+\frac{|S_P|}{\ell}\Bigr)\Bigl(1+\frac{|S_P|+\ell}{u}\Bigr)
\ \le\ \Bigl(1+\frac{|S_P|}{2}\Bigr)\Bigl(1+\frac{|S_P|+2}{|S_P|-4}\Bigr)\\[2pt]
&\ =\ \frac{|S_P|+2}{2}\cdot\frac{2|S_P|-2}{|S_P|-4}
=\frac{(|S_P|-1)(|S_P|+2)}{|S_P|-4} .
\end{align*}
It remains to check $(|S_P|-1)(|S_P|+2)/(|S_P|-4)\le 4|S_P|^{2}/(|S_P|-2)$. Clearing
denominators, this reads $(|S_P|-1)(|S_P|+2)(|S_P|-2)\le 4|S_P|^{2}(|S_P|-4)$, that is
$|S_P|^{3}-|S_P|^{2}-4|S_P|+4\le 4|S_P|^{3}-16|S_P|^{2}$, that is
\[
3|S_P|^{3}-15|S_P|^{2}+4|S_P|-4\ \ge\ 0 ,
\]
which holds because $3|S_P|^{2}(|S_P|-5)\ge0$ and $4|S_P|-4>0$ for $|S_P|\ge5$.

For $|S_P|\in\{3,4\}$ we argue directly from \eqref{eq:factor}. Using $\ab\ge1$ in
\eqref{eq:factor}, set
\[
F(\ell,u):=\Bigl(1+\frac{|S_P|}{\ell}\Bigr)\Bigl(1+\frac{|S_P|+\ell}{u}\Bigr)
\ \ge\ \frac{h|\Gamma_T|}{N} ,
\]
and observe that $F(\ell,u)$ is strictly decreasing in $u$; hence, for each $\ell\ge2$,
it is enough to replace $u$ by its smallest admissible value. Recall that $u\ge1$ by
\eqref{eq:u}, that $v\ge|S_P|-2$ by \eqref{eq:I1}, and that
$u\ge(\ell-1)(|S_P|-2)-2$ by Lemma~\ref{lem:u}.

Let $|S_P|=3$, so that $4|S_P|^{2}/(|S_P|-2)=36$. For $\ell=2$ and $\ell=3$ we only
use $u\ge1$, which gives
\[
F(2,1)=\tfrac52\cdot 6=15,\qquad F(3,1)=2\cdot 7=14 .
\]
For $\ell\ge4$, Lemma~\ref{lem:u} gives $u\ge\ell-3\ge1$, whence
\[
F(\ell,\ell-3)=\frac{\ell+3}{\ell}\cdot\frac{2\ell}{\ell-3}=\frac{2(\ell+3)}{\ell-3}
=2+\frac{12}{\ell-3} ,
\]
which decreases as $\ell$ grows and equals $14$ for $\ell=4$. Therefore
$h|\Gamma_T|/N\le15<36$.

Let $|S_P|=4$, so that $4|S_P|^{2}/(|S_P|-2)=32$. For $\ell=2$ we have $u=2v-4$ with
$v$ an integer, so that $u\ge1$ forces $v\ge3$, that is $u\ge2$, and
$F(2,2)=3\cdot 4=12$. For $\ell\ge3$, Lemma~\ref{lem:u} gives $u\ge2\ell-4\ge2$, whence
\[
F(\ell,2\ell-4)=\frac{\ell+4}{\ell}\cdot\frac{3\ell}{2\ell-4}=\frac{3(\ell+4)}{2(\ell-2)}
=\frac32+\frac{9}{\ell-2} ,
\]
which decreases as $\ell$ grows and equals $21/2$ for $\ell=3$. Therefore
$h|\Gamma_T|/N\le12<32$.

In both cases \eqref{eq:target} holds. These bounds are attained, namely at
$(\ell,v,u)=(2,2,1)$ for $|S_P|=3$ and at $(\ell,v,u)=(2,3,2)$ for $|S_P|=4$.
\end{proof}

From now on we assume that $\ell=1$. Then $\ab|\Gamma_T|=|S_P|+1$ and, by \eqref{eq:rel}, $\ab N=v-|S_P|$, so that
\begin{equation}\label{eq:l1}
v=|S_P|+\ab N,\qquad \ab h=v+1=|S_P|+1+\ab N ,
\end{equation}
and \eqref{eq:master} becomes
\begin{equation}\label{eq:l1split}
\frac{h|\Gamma_T|}{N}=\frac{(|S_P|+1+\ab N)(|S_P|+1)}{\ab^{2}N}
=\frac{(|S_P|+1)^{2}}{\ab^{2}N}+\frac{|S_P|+1}{\ab} .
\end{equation}
\begin{lemma}\label{lem:key}
Assume $\ell=1$. Then
\[
\ab^{2}N\ \ge\ \min\Bigl\{\,|S_P|-3,\ \frac{|S_P|+1}{2}\,\Bigr\} .
\]
More precisely, exactly one of the following holds.
\begin{itemize}
\item[(i)] $S_P^{(2)}=S_P$ and  $\ab N\ge|S_P|-3$;
\item[(ii)] $S_P^{(2)}\subsetneq S_P$ and $|S_P|<2\ab^{2}N$.
\end{itemize}
\end{lemma}
\begin{proof}
Since $\Spi{0}=\Spi{1}=S_P$, Hilbert's formula together with
\eqref{eq:l1} gives
\begin{equation}\label{eq:hilbert}
\sum_{i\ge2}\bigl(|\Spi{i}|-1\bigr)=(|S_P|+v)-2(|S_P|-1)=v-|S_P|+2=\ab N+2\ \ge\ 3 ,
\end{equation}
because $\mu+1\ge1$ and $N\ge1$. In particular the sum is nonzero, so
$\Spi{2}\ne 1$; the two alternatives are therefore $\Spi{2}=S_P$ and
$1<|S_P^{(2)}|<|S_P|$, and they are mutually exclusive.
\\
\emph{Case (i): $\Spi{2}=S_P$.} Then the term $i=2$ of the sum in
\eqref{eq:hilbert} alone equals $|S_P|-1$, whence $\ab N+2\ge|S_P|-1$, that is
$\ab N\ge|S_P|-3$; and $\mu+1\ge1$ gives $\ab^{2}N\ge\ab N\ge|S_P|-3$.
\\
\emph{Case (ii): $1<|S_P^{(2)}|<|S_P|$.} Now \eqref{eq:I2b} applies and yields
\begin{equation}\label{eq:t}
\frac{|S_P|}{|S_P^{(2)}|}-1=t\,h,\qquad t\ge 1 ,
\end{equation}
the integer $t$ being positive because $|S_P^{(2)}|<|S_P|$. Multiply \eqref{eq:t} by
$\mu+1$ and use $\ab h=|S_P|+1+\ab N$ from \eqref{eq:l1}:
\[
\ab\frac{|S_P|}{|S_P^{(2)}|}=\ab(th+1)=t\bigl(\ab h\bigr)+\ab
=
\]
\[=t\bigl(|S_P|+1+\ab N\bigr)+\ab=t|S_P|+\bigl(\ab+t(1+\ab N)\bigr),
\]
that is
\begin{equation}\label{eq:key}
\frac{|S_P|}{|S_P^{(2)}|}\bigl(\ab-t|S_P^{(2)}|\bigr)=\ab+t(1+\ab N) .
\end{equation}
The right-hand side is positive and $|S_P|/|S_P^{(2)}|>0$, so $\ab-t|S_P^{(2)}|>0$; being an
integer, $\ab-t|S_P^{(2)}|\ge1$, i.e.
\begin{equation}\label{eq:tq2}
t\,|S_P^{(2)}|\le\ab-1 .
\end{equation}
Since $t\ge1$ this gives $|S_P^{(2)}|\le\ab-1$, and since $|S_P^{(2)}|\ge p\ge2$ it gives
$\mu+1\ge3$.
Dividing \eqref{eq:key} by $\ab-t|S_P^{(2)}|\ge1$ and multiplying by
$|S_P^{(2)}|$,
\[
|S_P|\ \le\ |S_P^{(2)}|\bigl(\ab+t(1+\ab N)\bigr)=\ab|S_P^{(2)}|+t|S_P^{(2)}|(1+\ab N)
\ \le
\]
\[
\le \ab(\ab-1)+(\ab-1)(1+\ab N)=(\ab-1)\bigl(\ab+1+\ab N\bigr),
\]
using $|S_P^{(2)}|\le\ab-1$ and \eqref{eq:tq2}. Finally
\[
(\ab-1)(\ab+1+\ab N)=\ab^{2}-1+\ab^{2}N-\ab N\ <\ 2\ab^{2}N
\]
is equivalent to $\ab^{2}-1<\ab^{2}N+\ab N=\ab N(\ab+1)$, which holds because
$N\ge1$ gives $\ab N(\ab+1)\ge\ab(\ab+1)=\ab^{2}+\ab>\ab^{2}-1$. Hence
$|S_P|<2\ab^{2}N$, i.e.\ $\ab^{2}N\ge(|S_P|+1)/2$.
\end{proof}
\begin{lemma}\label{lem:l1}
If $\ell=1$ then \eqref{eq:target} holds.
\end{lemma}
\begin{proof}
Let $|S_P|\ge7$. We claim $\min\{|S_P|-3,(|S_P|+1)/2\}\ge(|S_P|+1)/2$ whenever $|S_P|\ge7$:
indeed $|S_P|-3\ge(|S_P|+1)/2$ is equivalent to $|S_P|\ge7$. So Lemma~\ref{lem:key}
gives $\ab^{2}N\ge(|S_P|+1)/2$, and \eqref{eq:l1split} with $\mu+1\ge1$ gives
\[
\frac{h|\Gamma_T|}{N}\ \le\ \frac{2(|S_P|+1)^{2}}{|S_P|+1}+(|S_P|+1)=3(|S_P|+1) .
\]
Now $3(|S_P|+1)\le 4|S_P|^{2}/(|S_P|-2)$ is equivalent to $3(|S_P|+1)(|S_P|-2)\le 4|S_P|^{2}$,
i.e.\ to $0\le|S_P|^{2}+3|S_P|+6$, which always holds.
\\For $|S_P|\in\{3,4\}$ use only $\ab^{2}N\ge1$ in \eqref{eq:l1split}: the bound is
$(|S_P|+1)^{2}+(|S_P|+1)$, that is $20$ and $30$, against $4|S_P|^{2}/(|S_P|-2)=36$ and
$32$. For $|S_P|=5$, case (i) of Lemma~\ref{lem:key} gives $\ab N\ge 2$ and case
(ii) gives $\ab^{2}N\ge3$, so in both $\ab^{2}N\ge2$; the bound is
$36/2+6=24$, against $100/3$.
\end{proof}

\noindent {\em Proof of Proposition \ref{prop:main}}. The claim follows from Lemmas \ref{easystaff}, \ref{lem:lge2}, and \ref{lem:l1}. \qed

\section{Proof of Theorem \ref{themain18122025}}
We summarize the results obtained in the
previous Sections to provide a proof of
Theorem \ref{themain18122025}.
Let $\cX$, $\gamma(\cX)$, $\Gamma$, $\Delta$, $p$, $P$, $S_P$ be as in the statement of Theorem \ref{themain18122025}.
Let
$\Delta_0$ be the set consisting of all points $Q$ of $\cX$, other than $P$, which are fixed by some non-trivial element of $S_P$. Note that $\Delta_0$ is non-empty by Lemma
\ref{lem24072024}. We assume that $|\Delta|>2$.

In Case (A), $\Delta$ is contained in $\{P\}\cup \Delta_0$. If the action of $S_P$ on $\Delta\setminus \{P\}$ is not transitive, then \eqref{eq18122025}
follows from Proposition \ref{pro16072025}, otherwise from Proposition \ref{11set26}.

In Case (B), we may assume by Proposition \ref{prop290723Ater} that $|\Gamma|>84(\gg(\cX)-1)$. Then one of the cases (a), (b), (c), (d) in Result \ref{res56.116} holds for $G=\Gamma$. If (a) holds, then \eqref{eq18122025} holds by Proposition \ref{prop290723Aprimo}.
If  (b)  holds then \eqref{eq18122025} follows from Proposition \ref{prop290723A}.
{\color{black} If (c) holds, then \eqref{eq18122025} follows from Proposition \ref{pro24072024}.}
Finally, in case (d)
if $S_P$ has at least two short orbits other than $\{P\}$ the claim follows from Proposition \ref{prop290723bis}, otherwise
 from Lemma \ref{lem14dic21} together with Proposition \ref{prop:main}.

\section{Sharpness of the bound and necessity of the hypotheses}
\label{exam}

\subsection{Example}\label{skms} For $p>2$,
let $\cY$ be a non-singular model of the function field $\mathbb{K}(x,y)$ where
$
y^p-y=x+\frac{1}{x}.
$
By Artin-Schreier theory (see e.g. \cite[Proposition 3.7.8]{stichtenoth1993}), $\mathbb K(\cY):\mathbb K(x)$ is a Galois extension of degree $p$, totally ramified at two places of $\mathcal Y$, say $P$ and $P'$, lying over $x=0$ and $x=\infty$ respectively. Also, $\ord_P(x)=p$ and $\ord_{P'}(x)=-p$, whereas $\ord_{Q}(x)=0$ for each place $Q$ distinct from $P$ and $P'$.

For $m>p$ with $\gcd(m,p)=1$ let $\cX_m$ be the non singular model of the function field $\mathbb K(y,t)$ where $t^m=x$ and hence
\begin{equation}
y^p-y=t^m+\frac{1}{t^m}.
\end{equation}

By Kummer theory (see e.g. \cite[Proposition 3.7.3]{stichtenoth1993}), $\mathbb K(\cX_m):\mathbb K(\cY)$ is a Galois extension of degree $m$, fully ramified at two places $\bar P$ and ${\bar{P'}}$ of $\mathcal X_m$, lying over $P$ and $P'$ respectively.

For $\tau$ ranging over $\mathbb{F}_{p}$, the maps $(t,y)\mapsto (t,y+\tau)$ belong to $\aut(\cX_m)$. They form a cyclic group $U$ of order $p$, and the quotient curve $\bar{\cX}_m=\cX_m/U$ is rational. The group $U$ preserves the set of poles of $y$, which is $\{\bar P,{\bar{P'}}\}$, and acts semiregularly on the other places of $\cX_m$. As $p$ is odd, $U$ actually fixes both $\bar P$ and ${\bar {P'}}$.
 Also, by the Deuring-Shafarevich formula, the $p$-rank of $\cX_m$ then is $p-1$.

Then both conditions (*) and (**) are satisfied by each $\mathcal X_m$.

When $\mu$ ranges over the set of $m$-th roots of unity in $\mathbb K$, the maps $(t,y)\mapsto (\mu t,y)$ belong to $\aut(\cX_m)$. They form a cyclic group $C$ of order $m$ such that $C\cap U$ is trivial.
Moreover, $\imath:=(t,y)\mapsto (1/t,y)$ is in $\aut(\cX_m)$ and $C$ together with $\imath$ generate a dihedral group $V$ of order $2m$. Since $\langle U,V \rangle=U\times V$, we end up with a subgroup of $\aut(\cX_m)$ of order $2mp$ which is the direct product of a cyclic group of order $p$ with a dihedral group of order $2m$.
Therefore, the
sizes of the automorphism groups of $\{\cX_m\mid m>p, \,(p,m)=1\}$ cannot be bounded by any function depending only on the $p$-rank, despite (*) and (**) being satisfied.
This does not contradict Theorem \ref{themain18122025} since $U \times V$ preserves the set $\{\bar P, {\bar{P'}}\}$.

\subsection{Example}\label{exfermat}
For every odd integer $r\geq 3$, put
\(
q=p^{r}
\)
and let $\cH_q$ be the Hermitian curve
\[
\cH_q:\qquad x^{q+1}+y^{q+1}+1=0.
\]
As $\cH_q$ is a non-singular plane curve we have
$
\gg(\cH_q)=\frac{q(q-1)}{2}.
$
Also,
$\gamma(\cH_q)=0$ (see e.g. \cite[Lemma 12.1(g)]{HKT}).

Since $q$ is an odd power of $p$, we have that
$(q+1)/(p+1)$ is an odd integer.
Let $\cY_q$ be the nonsingular model of the function
field $\mathbb K(x,y,z)$ defined by
\begin{equation}
\label{eq10092026}
x^{q+1}+y^{q+1}+1=0,
\qquad
(x^{p+1}+1)(z^p-z)=1.
\end{equation}

We determine the ramification of the Artin--Schreier cover
$\cY_q\longrightarrow\cH_q$. Since $(q+1)/(p+1)$ is odd, the
zeros on $\cH_q$ of $x^{p+1}+1$ are precisely the $p+1$ points
\(
P_\xi=(\xi,0),
\,
\xi^{p+1}=-1.
\)
At each $P_\xi$, the function $y$ is a local parameter and the
equation of $\cH_q$ gives
\(
\ord_{P_\xi}(x-\xi)=q+1
\).
Since $\xi$ is a simple zero of $x^{p+1}+1$, we obtain
$
\ord_{P_\xi}(x^{p+1}+1)=q+1.
$
Thus $1/(x^{p+1}+1)$ has exactly $p+1$ poles, all of order $q+1$.
Since $p\nmid(q+1)$, the second equation in
\eqref{eq10092026} defines an Artin--Schreier extension of degree
$p$, totally ramified over the points $P_\xi$ and unramified
elsewhere.
Let
\[
S=\{(x,y,z)\mapsto(x,y,z+b):b\in\mathbb F_p\}.
\]
Then $S$ has order $p$, $\cY_q/S\cong\cH_q$, and its short orbits are
the $p+1$ points of $\cY_q$ lying over the points $P_\xi$, each of
which is fixed by $S$. The Deuring--Shafarevich formula yields
\[
\gamma(\cY_q)-1
 =p\bigl(\gamma(\cH_q)-1\bigr)+(p+1)(p-1)
 =-p+(p+1)(p-1),
\]
that is,
\begin{equation}
\label{eq:prank-no-star-example}
\gamma(\cY_q)=p(p-1)>0,
\end{equation}
independently of the exponent $r$.

For $\lambda^{q+1}=1$, $\mu^{p+1}=1$, and $b\in\mathbb F_p$, the map
\[
\sigma_{\lambda,\mu,b}:(x,y,z)\longmapsto
   (\mu x,\lambda y,z+b)
\]
is an automorphism of $\cY_q$.
These automorphisms form a subgroup
\(
\Gamma\cong C_{q+1}\times C_{p+1}\times C_p
\)
of $\aut(\cY_q)$, and therefore
\begin{equation}
\label{eq:group-no-star-example}
|\Gamma|=p(p+1)(q+1).
\end{equation}

For each $\xi^{p+1}=-1$, denote by $\widetilde P_\xi$ the unique point
of $\cY_q$ lying over $P_\xi$. The subgroup of $\Gamma$ consisting of maps $\sigma_{1,\mu,0}$ with $\mu^{p+1}=1$ is transitive on these points. Hence, the
$\Gamma$-orbit $\Delta$ of a fixed $\widetilde P_\xi$
has length
\[
|\Delta|=p+1>2.
\]
The stabilizer of a point $\widetilde P_\xi$ is
\[
\Gamma_{\widetilde P_\xi}
 =\{\sigma_{\lambda,1,b}:\lambda^{q+1}=1,\
                              b\in\mathbb F_p\}
 \cong C_{q+1}\times C_p,
\]
whose Sylow $p$-subgroup is $S$. However,
\[
\cY_q/S\cong\cH_q
\]
is not rational. Thus hypothesis
{\rm (*)} does not hold. Notice also that $\Delta$ is precisely the
set of points fixed by the nontrivial elements of $S$; in particular,
hypothesis {\rm (***)} does hold.

Finally, for fixed $p$, equations
\eqref{eq:prank-no-star-example} and
\eqref{eq:group-no-star-example} show that
$\gamma(\cY_q)=p(p-1)$ remains constant, whereas $|\Gamma|$ tends to
infinity as $r$ (or, equivalently, $q$) tends to infinity. Therefore, if hypothesis {\rm (*)}
is omitted, no upper bound on $|\Gamma|$ depending only on the
$p$-rank can hold, even when $\gamma(\cY_q)>0$ and $|\Delta|>2$.

\subsection{Example}\label{ex3}

Throughout this subsection $q=p^{r}$, and $\mathcal{H}_q$ denotes the Hermitian
curve, written in the form
\begin{equation}\label{eq:herm}
  y^{q+1}=x^{q}-x .
\end{equation}

The line $y=0$  meets $\mathcal{H}_q$ in the $q$
affine points $(\beta,0)$ with $\beta\in\mathbb{F}_{q}$ together with the unique
point $P_{\infty}$ at infinity. Their images under $x$ are the $q+1$ points of
$\mathbb{P}^{1}(\mathbb{F}_{q})$, and each of them is totally ramified in the
cover $\mathcal{H}_q\to\mathbb{P}^{1}(\mathbb K)$, since
$\ord(x^{q}-x)\in\{1,-q\}$ there and both values are prime to $q+1$.

\medskip
Let $\mathcal{A}$ be the set of all maps
\begin{equation}\label{eq:autherm}
  \sigma_{A,\delta}\colon (x,y)\longmapsto
  \Bigl(\frac{ax+b}{cx+d},\ \frac{\delta\,y}{cx+d}\Bigr),
  \qquad
  A=\begin{pmatrix} a & b\\ c& d\end{pmatrix}\in \mathrm{GL}(2,q),
  \quad \delta^{q+1}=\det A .
\end{equation}

\begin{lemma}\label{lem:Agroup}
$\mathcal{A}$ is a subgroup of $\mathrm{Aut}(\mathcal{H}_q)$ of order
$q(q-1)(q+1)^{2}$. Its subgroup $C_{q+1}=\{(x,y)\mapsto(x,\lambda y):
\lambda^{q+1}=1\}$ is normal, it is the kernel of the action of $\mathcal{A}$ on
$\mathbb{K}(x)$, and $\mathcal{A}/C_{q+1}\cong \mathrm{PGL}(2,q)$ in its natural
action on $\mathbb{P}^{1}(\mathbb{F}_{q})$.
\end{lemma}

\begin{proof}
For $a,b,c,d\in\mathbb{F}_{q}$ with $\det A=ad-bc\neq0$,
\[
  \Bigl(\frac{ax+b}{cx+d}\Bigr)^{q}-\frac{ax+b}{cx+d}
  =\frac{(ax^{q}+b)(cx+d)-(ax+b)(cx^{q}+d)}{(cx+d)^{q+1}}
  =\det A\cdot\frac{x^{q}-x}{(cx+d)^{q+1}} ,
\]
so $\sigma_{A,\delta}$ preserves \eqref{eq:herm} precisely when
$\delta^{q+1}=\det A$; such a $\delta$ exists as $\mathbb{K}$ is algebraically
closed. Composition corresponds to multiplication in $\mathrm{GL}(2,q)$ together
with multiplication of the $\delta$'s, whence $\mathcal{A}$ is a group. The map
$\sigma_{A,\delta}\mapsto A\,\mathbb{F}_{q}^{*}$ is a surjective homomorphism
onto $\mathrm{PGL}(2,q)$ whose kernel consists of the maps fixing $x$, that is,
of $C_{q+1}$. Therefore
$|\mathcal{A}|=(q+1)\,|\mathrm{PGL}(2,q)|=q(q-1)(q+1)^{2}$.
\end{proof}

Let
\begin{equation}\label{eq:uu}
  u=\frac{(x^{q^{2}}-x)^{q+1}}{(x^{q}-x)^{q^{2}+1}}\ \in\ \mathbb{K}(x),
\end{equation}

\begin{remark}\label{rem:dickson}
Since $a,b,c,d\in\mathbb{F}_{q}\subset\mathbb{F}_{q^{2}}$, the computation carried out in
the proof of Lemma \ref{lem:Agroup} applies verbatim with $q^{2}$ in place of $q$. Writing
$\lambda_A(x)=(ax+b)/(cx+d)$ for $A\in \mathrm{GL}(2,q)$, this gives both
\[
\lambda_A(x)^{q}-\lambda_A(x)=\det A\cdot\frac{x^{q}-x}{(cx+d)^{q+1}},
\qquad
\lambda_A(x)^{q^{2}}-\lambda_A(x)=\det A\cdot\frac{x^{q^{2}}-x}{(cx+d)^{q^{2}+1}} .
\]
Hence, for $u$ as in \eqref{eq:uu}, the powers of $cx+d$ cancel and
\[
u\bigl(\lambda_A(x)\bigr)
=\frac{(\det A)^{q+1}\bigl(x^{q^{2}}-x\bigr)^{q+1}(cx+d)^{-(q^{2}+1)(q+1)}}
      {(\det A)^{q^{2}+1}\bigl(x^{q}-x\bigr)^{q^{2}+1}(cx+d)^{-(q+1)(q^{2}+1)}}
=(\det A)^{-q(q-1)}\,u(x)=u(x),
\]
because $\det A\in\mathbb{F}_{q}^{*}$. Thus $u$ is $\mathrm{PGL}(2,q)$-invariant, so that
$\mathbb{K}(u)\subseteq\mathbb{K}(x)^{\mathrm{PGL}(2,q)}$.
\end{remark}

Let $\mathcal{X}_q$ be a non-singular model of the
irreducible algebraic space curve with affine equations
\begin{equation}\label{eq:X}
  y^{q+1}=x^{q}-x,\qquad z^{p}-z=u.
\end{equation}

\newcommand{\wpop}{\wp}

\begin{proposition}\label{prop:sharp}
Let $q\ge 4$. Then
the curve $\mathcal{X}_q$ of \eqref{eq:X} has the following
properties.
\begin{enumerate}
\item[(i)] $\gamma(\mathcal{X}_q)=q(p-1)$ and
\[
  2\gg(\mathcal{X}_q)-2=p\bigl(q^{2}-q-2\bigr)
  +(q+1)(p-1)\bigl((q+1)(q-1)^{2}+1\bigr).
\]
\item[(ii)] $\mathrm{Aut}(\mathcal{X}_q)$ contains a subgroup
$\Gamma=\mathcal{A}\times W$, where $W=\langle (x,y,z)\mapsto(x,y,z+1)\rangle$,
of order
\[
  |\Gamma|=p\,q(q-1)(q+1)^{2}=p(q^{3}-q)(q+1).
\]
\item[(iii)] Both $(*)$ and $(\!*\!*\!*\!)$ hold for $\Gamma$, the orbit $\Delta$
consisting of the $q+1$ points of $\mathcal{X}_q$ lying over the points of
$\mathcal{H}_q$ on the chord $y=0$; moreover $|S_{P}|=pq$ for $P\in\Delta$.
\item[(iv)] Equality holds in \eqref{eqB260223}, that is
$|\Delta|=\frac{1}{p-1}\gamma(\mathcal{X}_q)+1$, and
\begin{equation}\label{eq:sharp}
  |\Gamma|\ >\ p\,q^{4}
  \ =\ \frac{1}{p^{3}}\Bigl(\frac{p}{p-1}\Bigr)^{4}\gamma(\mathcal{X}_q)^{4}.
\end{equation}
\end{enumerate}
\end{proposition}

\begin{proof}
Write $\Lambda=\mathbb{P}^{1}(\mathbb{F}_{q})$ and let
$\Delta_{\mathcal{H}}$ be the set of the $q+1$ points of $\mathcal{H}_q$ lying over
$\Lambda$, that is, the points of $\mathcal{H}_q$ on the chord $y=0$.
Let $u$ be as in \textup{(32)}.
The poles of $u$ in $\mathbb{K}(x)$ are exactly the $q+1$ places of $\Lambda$.
Put
\[
\varphi=(x^{q}-x)^{1-q}\in\mathbb{K}(x).
\]
We prove that
\begin{equation}\label{eq:u-in-t}
u=\frac{(1+\varphi)^{q+1}}{\varphi}=\frac1\varphi+1+\varphi^{q-1}+\varphi^{q}.
\end{equation}
Write $w=x^{q}-x$, so that $w^{q}=x^{q^{2}}-x^{q}$ and hence $x^{q^{2}}-x=w^{q}+w$. Therefore
\[
u=\frac{(w^{q}+w)^{q+1}}{w^{q^{2}+1}}
=\frac{w^{q+1}\bigl(w^{q-1}+1\bigr)^{q+1}}{w^{q^{2}+1}}
=\frac{\bigl(w^{q-1}+1\bigr)^{q+1}}{w^{q^{2}-q}} .
\]
Since $w^{q-1}=\varphi^{-1}$ and $w^{q^{2}-q}=\bigl(w^{q-1}\bigr)^{q}=\varphi^{-q}$, this is
$u=\varphi^{q}\,(1+\varphi^{-1})^{q+1}=(1+\varphi)^{q+1}/\varphi$; and
$(1+\varphi)^{q+1}=(1+\varphi)^{q}(1+\varphi)=(1+\varphi^{q})(1+\varphi)$ gives \eqref{eq:u-in-t}.

Now, the places of $\Lambda$ are the $q$ zeros of $w$, at which $\ord(\varphi)=-(q-1)$, together with
the place $x=\infty$, at which $\ord_{\infty}(w)=-q$ and so $\ord_{\infty}(\varphi)=q(q-1)$. By
\eqref{eq:u-in-t} the pole divisor of $u$  in $\mathbb P^1(\mathbb K)$ is thus $q(q-1)$ times each zero of $w$
(coming from $\varphi^{q}$) plus $q(q-1)$ times $x=\infty$ (coming from $1/\varphi$), whence
$\deg u=(q+1)\,q(q-1)=q^{3}-q$.
This, together with Remark \ref{rem:dickson}, implies that $\mathbb{K}(u)=\mathbb{K}(x)^{\mathrm{PGL}(2,q)}$.

Note that $\mathrm{PGL}(2,q)$ is transitive on $\Lambda$ and fixes $u$; therefore, it
suffices consider the place $\bar O$ of $\mathbb K(x)$ centered at the origin, where $\ord_{\bar O}(\varphi)=-(q-1)$.
By \eqref{eq:u-in-t},
\[
u-(\psi^p-\psi)=\frac1\varphi+1+\varphi+\varphi^{q-1}
\]
with $\psi=\varphi+\varphi^p+\ldots \varphi^{q/p}$.
The four summands have valuation at $\bar O$ equal to $q-1$, $0$, $-(q-1)$ and
$-(q-1)^{2}$ respectively; as $q\ge3$ these are pairwise distinct, so no cancellation
occurs and the pole order of the right-hand side is exactly $(q-1)^{2}$. Finally
$q-1\equiv-1\pmod p$ gives $(q-1)^{2}\equiv1\pmod p$, so $(q-1)^{2}$ is prime to $p$ and is
therefore the conductor at $\bar O$.
  As every place of
  $\Lambda$ is totally ramified with
  ramification index $q+1$, prime to $p$, the conductor at the corresponding
  point $P\in\Delta_{\mathcal H}$ equals $M_P=(q+1)(q-1)^2$, which is prime to $p$.
In particular $z^{p}-z=u$ is a non-degenerate
Artin--Schreier extension of $\mathbb{K}(\mathcal{H}_q)$, ramified exactly at the
$q+1$ points of $\Delta_{\mathcal{H}}$, each of them totally ramified, and
unramified elsewhere. Let $\Delta$ be the set of the $q+1$ points of
$\mathcal{X}_q$ lying over $\Delta_{\mathcal{H}}$.

\emph{(i).} Apply the Deuring--Shafarevich formula to the cyclic group of order $p$ $W=\langle (x,y,z)\mapsto(x,y,z+1)\rangle$, whose short orbits are
the $q+1$ points of $\Delta$, each of size one, and whose quotient curve is
$\mathcal{H}_q$. Then
\[
  \gamma(\mathcal{X}_q)-1=p(0-1)+(q+1)(p-1),
  \qquad\text{that is}\qquad \gamma(\mathcal{X}_q)=q(p-1).
\]
The genus follows from the Riemann--Hurwitz formula applied to
$\mathcal{X}_q\to\mathcal{H}_q$, the different exponent at each point of $\Delta$
being $(p-1)(M_{P}+1)$ by \cite[Proposition 3.7.8]{stichtenoth1993}.

\emph{(ii).} Every $\sigma\in\mathcal{A}$ acts on $\mathbb{K}(x)$ through
$\mathrm{PGL}(2,q)$ by Lemma \ref{lem:Agroup}, and $u$ is
$\mathrm{PGL}(2,q)$-invariant; hence $\sigma(u)=u$ and $\sigma$ extends to
$\mathcal{X}_q$ by $z\mapsto z$. This extension is an injective homomorphism, its
image commutes with $W$ elementwise and meets $W$ trivially, so
$\Gamma=\mathcal{A}\times W\le\mathrm{Aut}(\mathcal{X}_q)$ has order
$p\,|\mathcal{A}|$, and Lemma \ref{lem:Agroup} gives the stated value.

\emph{(iii).} Let $P\in\Delta$ lie over $x=\infty$ and let
\[
  U=\{(x,y,z)\mapsto (x+\alpha,y,z)\ :\ \alpha\in\mathbb{F}_{q}\}.
\]
Then $U\le\Gamma_{P}$, as $(x+\alpha)^{q}-(x+\alpha)=x^{q}-x$ and $u$ is
invariant; $U$ is elementary abelian of order $q$ and $U\cap W=1$, so
$U\times W\le \Gamma_{P}$ has order $pq$. Since $p\nmid (q-1)(q+1)^{2}$, the
$p$-part of $|\Gamma|$ is $pq$, whence $S_{P}=U\times W$ is a Sylow
$p$-subgroup of $\Gamma$. Now
\[
  \mathbb{K}(\mathcal{X}_q)^{S_{P}}
  =\bigl(\mathbb{K}(\mathcal{X}_q)^{W}\bigr)^{U}
  =\mathbb{K}(\mathcal{H}_q)^{U}=\mathbb{K}(y),
\]
the last equality because $y$ is $U$-invariant and
$[\mathbb{K}(\mathcal{H}_q):\mathbb{K}(y)]=q=|U|$. Therefore
$\mathcal{X}_q/S_{P}$ is rational and $(*)$ holds.

For $(***)$, note that $W$ fixes $\Delta$ pointwise, while an element of
$S_{P}$ with non-trivial $U$-component induces on $\mathbb{P}^{1}(\mathbb K)$ a
non-trivial translation, whose only fixed place is $x=\infty$; such an element
therefore fixes no point of $\mathcal{X}_q$ other than $P$. Hence $\Delta$ is
exactly the set of points of $\mathcal{X}_q$ fixed by some non-trivial element of
$S_{P}$. Finally $\Delta$ is a $\Gamma$-orbit, because $\mathcal{A}$ induces
$\mathrm{PGL}(2,q)$ on $\Lambda$, which is transitive.

\emph{(iv).} From (i), $\frac{1}{p-1}\gamma(\mathcal{X}_q)+1=q+1=|\Delta|$. Moreover
$q(q-1)(q+1)^{2}=q^{4}+q^{3}-q^{2}-q>q^{4}$ for $q\ge2$, which gives the first
inequality in \eqref{eq:sharp}; the equality there is
$\bigl(\frac{p}{p-1}\bigr)^{4}\gamma(\mathcal{X}_q)^{4}=p^{4}q^{4}$.
\end{proof}

\begin{remark}\label{rem:sharp}
Comparison of \eqref{eq:sharp} with \eqref{eq18122025} shows that the bound of
Theorem \ref{themain18122025} is attained up to the factor $24p^{3}$. Against the
intermediate bound of Proposition \ref{11set26}(${\mathfrak{II}}$) the example is
considerably closer: since $\gamma(\mathcal{X}_q)=q(p-1)$,
\[
  \frac{|\Gamma|}
       {8\gamma(\mathcal{X}_q)\bigl(\frac{\gamma(\mathcal{X}_q)}{p-1}+1\bigr)^{3}}
  =\frac{p\,(q-1)}{8(p-1)(q+1)}\ \longrightarrow\ \frac{p}{8(p-1)}
  \qquad (q\to\infty),
\]
so that  Proposition \ref{11set26}(${\mathfrak{II}}$) is sharp up to an absolute constant,
independent of both $p$ and $q$. By (iv) of Proposition \ref{prop:sharp}, the
bound \eqref{eqB260223} of Lemma \ref{leA260223} is attained as well.
\end{remark}

{\color{black}
\subsection{Example}\label{caseB}

The curves of Sections \ref{skms}, \ref{exfermat} and \ref{ex3} all satisfy (***). We exhibit here a family of curves which satisfies (*), (**) and $|\Delta|>2$ but not (***), so that Case (B) is not empty.

For $p>2$, let $\cM$ be the Artin-Mumford curve, that is, a non-singular model of the function field $\mathbb{K}(x,y)$ where
\begin{equation}
\label{eqASM}
(x^p-x)(y^p-y)=1.
\end{equation}
By a result of Valentini and Madan \cite{VM}, $\aut(\cM)=(C_p\times C_p)\rtimes D$ with $D$ dihedral of order $2(p-1)$, generated by the maps
$$\tau_{a,b}:(x,y)\mapsto (x+a,y+b),\quad \mu:(x,y)\mapsto (y,x),\quad \theta_d:(x,y)\mapsto (dx,d^{-1}y),$$
with $a,b\in\mathbb{F}_p$ and $d\in\mathbb{F}_p^*$; see also \cite[Section 3]{AK}. Moreover $\gg(\cM)=\gamma(\cM)=(p-1)^2$, so that $\cM$ is ordinary, and the quotient curves $\cM/\langle \tau_{1,0}\rangle$ and $\cM/\langle \tau_{0,1}\rangle$ are rational; see \cite[Proposition 3.1]{AK}.

Put $U_x=\{\tau_{a,0}\mid a\in\mathbb{F}_p\}$ and $U_y=\{\tau_{0,b}\mid b\in\mathbb{F}_p\}$, so that $U_x\times U_y=C_p\times C_p$ and $\mathbb{K}(\cM)^{U_x}=\mathbb{K}(y)$. Over $\mathbb{K}(y)$, equation (\ref{eqASM}) reads $x^p-x=1/(y^p-y)$, whose right-hand side has $p$ simple poles, namely the places $y=\beta$ with $\beta\in\mathbb{F}_p$, and no other pole; hence the cover $\cM$ over $\cM/U_x$ is totally ramified at the $p$ places $P_\beta$ of $\cM$ lying over $y=\beta$, and unramified elsewhere. Symmetrically, the ramified places of $\cM$ over $\cM/U_y$ are the $p$ places $Q_\alpha$ lying over $x=\alpha$ with $\alpha\in\mathbb{F}_p$.

\begin{proposition}
\label{proASM}
Let $p>2$, let $\Gamma=\aut(\cM)$, and let $P=P_0$. Then $\Gamma_P=U_x\rtimes C_{p-1}$ and $S_P=U_x$; hypotheses {\rm{(*)}} and {\rm{(**)}} hold, the $\Gamma$-orbit of $P$ is $\Delta=\{P_\beta\}\cup\{Q_\alpha\}$ of size $2p>2$, and $\Delta_0=\{P_\beta\mid \beta\ne 0\}$, so that {\rm{(***)}} does not hold. Therefore $\cM$ is in Case {\rm{(B)}}, and
$$|\Gamma|=2p^2(p-1)=2\left(\sqrt{\gamma(\cM)}+1\right)^2\sqrt{\gamma(\cM)}.$$
\end{proposition}
\begin{proof}
The map $\tau_{0,b}$ takes $P_\beta$ to $P_{\beta+b}$, and $\mu$ takes $P_\beta$ to $Q_\beta$; hence $\Delta$ is a $\Gamma$-orbit of size $2p$. The map $\tau_{a,b}\theta_d$ takes $P_\beta$ to the place lying over $y=d^{-1}\beta+b$, so that it fixes $P_0$ if and only if $b=0$, while every element of $\Gamma$ involving $\mu$ takes each $P_\beta$ to some $Q_\alpha$. Therefore $\Gamma_P=U_x\rtimes C_{p-1}$ has order $p(p-1)$, and $S_P=U_x$ as $p\nmid p-1$. Now, (*) holds because $\cM/U_x=\cM/\langle \tau_{1,0}\rangle$ is rational, while (**) holds as $\gamma(\cM)=(p-1)^2>0$; also $\gg(\cM)=(p-1)^2\ge 2$ by $p>2$. Since the points fixed by some non-trivial element of $S_P=U_x$ are exactly the places $P_\beta$, we have $\Delta_0=\{P_\beta\mid \beta\ne 0\}$, whence $Q_0\in\Delta\setminus(\{P\}\cup\Delta_0)$ and (***) fails. Consistently, $S_{Q_0}=U_y$ and therefore $S_P\cap S_{Q_0}=U_x\cap U_y=1$. Finally, $\gamma(\cM)=(p-1)^2$ gives the displayed equality.
\end{proof}

We point out that $\gg(\cM)=(|S_P|-1)^2$, so that the bound of Lemma \ref{lem14dic21} is attained. We also point out that $U_x\times U_y$ is a normal $p$-subgroup of $\Gamma$ which fixes no point of $\Delta$; this is necessarily so, as a normal $p$-subgroup of $\Gamma$ fixing a point of $\Delta$ is contained in $S_P\cap S_Q$ for every $Q\in\Delta$, and hence is trivial in Case (B).
}

\subsection{Example}\label{exzieve}
The example of Section \ref{caseB} lives in odd characteristic and has bounded $p$-rank for
fixed $p$. We exhibit here a family in characteristic $2$ with $\gamma(\cX)\to\infty$ and
with a much larger automorphism group.

Let $q=2^{n}$ with $n\ge1$, and let $\mathcal{Z}_q$ be the \emph{Zieve curve}, that is, a
non-singular model of $\mathbb{K}(x,y,z)$ with
\begin{equation}\label{eq:zieve}
  y^{q}+y=\frac{1}{x^{q}+x},
  \qquad
  z^{q}+z=\frac{x^{q+1}}{x^{q}+x};
\end{equation}
see \cite{Zieve} and \cite[Section 4.1]{BNZ}. We collect from \cite{BNZ} the facts we need.
The curve $\mathcal{Z}_q$ is ordinary with
\begin{equation}\label{eq:zieve-inv}
  \gg(\mathcal{Z}_q)=\gamma(\mathcal{Z}_q)=q^{3}-q^{2}-q+1=(q-1)(q^{2}-1),
\end{equation}
and $\aut(\mathcal{Z}_q)$ contains $G=E\rtimes H$ where $H\cong\PGL(2,q)$ and
$E=\gal(\mathbb{K}(\mathcal{Z}_q)|\mathbb{K}(x))
  =\{\sigma_{\ga,\gb}:(x,y,z)\mapsto(x,y+\ga,z+\gb)\,:\,\ga,\gb\in\mathbb{F}_q\}$,
so that
\begin{equation}\label{eq:zieve-group}
  |G|=q^{3}(q^{2}-1);
\end{equation}
moreover $G=\aut(\mathcal{Z}_q)$ for $q>8$ \cite[Theorem 5.5]{BNZ}. Here $H$ is generated
by $\pi:(x,y,z)\mapsto(1/x,z,y)$ together with $q$ of the maps of the elementary abelian
group
\begin{equation}\label{eq:zieveH1}
  H_1=\{\gt_{\mu,\nu}:(x,y,z)\mapsto(x+\mu,\,y,\,z+\mu^{2}y+\nu)
        \,:\,\mu\in\mathbb{F}_q,\ \nu^{q}+\nu=\mu\},\qquad |H_1|=q^{2}.
\end{equation}
Finally, $G$ has exactly two short orbits \cite[Section 5.2]{BNZ}: the non-tame orbit
$\Omega_1$ consisting of the $q(q+1)$ points lying over the places of $\mathbb{K}(x)$ in
$\mathbb{P}^{1}(\mathbb{F}_q)$, and the tame orbit $\Omega_2$ consisting of the
$q^{3}(q-1)$ points lying over $\mathbb{P}^{1}(\mathbb{F}_{q^{2}})\setminus
\mathbb{P}^{1}(\mathbb{F}_q)$, with $|G_Q|=q+1$ for $Q\in\Omega_2$.

\begin{proposition}\label{prop:zieve}
Let $q=2^{n}$ with $n\ge1$, let $\Gamma=G$, and let $P$ be the point of $\mathcal{Z}_q$
lying over $x=\infty$ with $y(P)=0$. Then $\Gamma_P=H_1\rtimes C_{q-1}$ and $S_P=H_1$;
hypotheses {\rm(*)} and {\rm(**)} hold, $\Delta=\Omega_1$ has size $q^{2}+q>2$, and
$\{P\}\cup\Delta_0$ consists of the $q$ points lying over $x=\infty$, so that {\rm(***)}
does not hold. Therefore $\mathcal{Z}_q$ is in Case {\rm(B)}, and
\begin{equation}\label{eq:zieve-sharp}
  |\Gamma|=q^{3}(q^{2}-1)>\gamma(\mathcal{Z}_q)^{5/3}.
\end{equation}
\end{proposition}

\begin{proof}
As $\ord_{\infty}\bigl(1/(x^{q}+x)\bigr)=q>0$, the place $x=\infty$ splits completely in
$\mathbb{K}(x,y)|\mathbb{K}(x)$, the $q$ places above it being those with $y\in\mathbb{F}_q$;
as $\ord_{\infty}\bigl(x^{q+1}/(x^{q}+x)\bigr)=-1$, each of them is totally ramified in
$\mathbb{K}(x,y,z)|\mathbb{K}(x,y)$. Hence $\Omega_1$ contains exactly $q$ points over
$x=\infty$, and $E_z=\{\sigma_{0,\gb}\}$ is the inertia group of each of them. Since
$\Omega_1$ is a $\Gamma$-orbit, $\Delta=\Omega_1$ and $|\Gamma_P|=|\Gamma|/|\Omega_1|
=q^{2}(q-1)$ by the Orbit theorem; as $q-1$ is odd, $|S_P|=q^{2}$. Every
$\gt_{\mu,\nu}\in H_1$ fixes $y$, hence fixes $P$; therefore $S_P=H_1$ by
\eqref{eq:zieveH1}.

Now $y$ is $H_1$-invariant, and $[\mathbb{K}(\mathcal{Z}_q):\mathbb{K}(y)]=q^{2}=|S_P|$
since $[\mathbb{K}(\mathcal{Z}_q):\mathbb{K}(x,y)]=q$ while $x^{q}+x=1/(y^{q}+y)$ gives
$[\mathbb{K}(x,y):\mathbb{K}(y)]=q$. Thus
$\mathbb{K}(\mathcal{Z}_q)^{S_P}=\mathbb{K}(y)$, that is, $\mathcal{Z}_q/S_P$ is rational
and (*) holds; (**) holds by \eqref{eq:zieve-inv}, which also gives
$\gg(\mathcal{Z}_q)\ge3$.

For (***), let $\gt_{\mu,\nu}\in H_1$ be non-trivial. If $\mu\ne0$, then $\gt_{\mu,\nu}$
induces on $\mathbb{P}^{1}(\mathbb{K})$ the translation $x\mapsto x+\mu$, whose only fixed
place is $x=\infty$. If $\mu=0$, then $\nu\in\mathbb{F}_q^{*}$ and
$\gt_{0,\nu}=\sigma_{0,\nu}\in E_z$, whose fixed points are exactly the $q$ points over
$x=\infty$; indeed, the inertia group in $E$ of a point over $x=\gl$ with
$\gl\in\mathbb{F}_q$ is $\{\sigma_{\ga,\gl^{2}\ga}\}$, which meets $E_z$ trivially. Hence
$\{P\}\cup\Delta_0$ is the set of the $q$ points over $x=\infty$, whereas
$|\Delta|=q^{2}+q>q$; so (***) fails. Consistently, $Q=\pi(P)$ lies over $x=0$,
$S_Q=\pi S_P\pi^{-1}$, and $S_P\cap S_Q=1$.

Finally, \eqref{eq:zieve-sharp} reads $q^{9}>(q-1)^{7}(q+1)^{2}$ by
\eqref{eq:zieve-inv}, and this holds for every $q\ge2$ since
$(1-t)^{7}(1+t)^{2}<1$ for $0<t\le\ha$.
\end{proof}

\begin{remark}\label{rem:zieve}
\begin{itemize}
\item[\rm(i)] Case (B) is thus not empty in characteristic $2$, nor is it vacuous with
$\gamma(\cX)$ arbitrarily large in a fixed characteristic.
\item[\rm(ii)] By \eqref{eq:zieve-sharp}, $|\Gamma|$ grows like
$\gamma(\mathcal{Z}_q)^{5/3}$, and in fact
$|\Gamma|/\gamma(\mathcal{Z}_q)^{5/3}\to1$ as $n\to\infty$, whereas the example of
Section \ref{caseB} only gives $|\Gamma|\sim2\gamma(\cX)^{3/2}$.
\item[\rm(iii)] Bound \eqref{eq:I1} is attained. Indeed $\ord_P(z)=-1$, so that $1/z$ is a
local parameter at $P$, and for $\gt_{\mu,\nu}\in S_P$ non-trivial we have
$\ord_P(\mu^{2}y+\nu)=0$, as $y(P)=0$ and $\nu=0$ forces $\mu=\nu^{q}+\nu=0$; hence
$\ord_P\bigl(\gt_{\mu,\nu}(1/z)-1/z\bigr)=2$. Therefore all non-trivial elements of $S_P$
have ramification number $1$ at $P$, whence
$d_P(\Gamma)=(|\Gamma_P|-1)+(|S_P|-1)=q^{3}-2$ and $v=|S_P|-2$. By contrast, the bound
$\gg(\cX)\le(|S_P|-1)^{2}$ of Lemma \ref{lem14dic21} is attained by the curve of Section
\ref{caseB} but not here, as $(q-1)(q^{2}-1)<(q^{2}-1)^{2}$; consistently with Lemma
\ref{lem14dic21}, $\mathbb{K}(\mathcal{Z}_q)=\mathbb{K}(y,z)$ with
$\mathbb{K}(y)=\mathbb{K}(\mathcal{Z}_q)^{S_P}$ and
$\mathbb{K}(z)=\mathbb{K}(\mathcal{Z}_q)^{S_Q}$.
\item[\rm(iv)] Case (d) of Result \ref{res56.116} occurs, and $S_P$ fixes each of the $q$
points over $x=\infty$, so it has $q-1$ short orbits other than $\{P\}$. Hence
Proposition \ref{prop290723bis} applies whenever $|\Gamma|>84(\gg(\mathcal{Z}_q)-1)$, that
is, for $q\ge16$; for $q\le8$ the bound follows from Proposition \ref{prop290723Ater}.
\end{itemize}
\end{remark}

\end{document}